\documentclass[reqno]{amsart}
\usepackage{amssymb,latexsym}
\usepackage{amsmath}
\usepackage{amsthm}
\usepackage{graphicx}
\usepackage{hyperref}
\usepackage{titletoc}
\usepackage{cancel}
\usepackage{color,xcolor}
\usepackage{epsfig}
\usepackage{epstopdf}
\usepackage{bm}

\usepackage[pagewise]{lineno}

\numberwithin{equation}{section}
\newtheorem{theorem}{Theorem}[section]

\newtheorem{lemma}[theorem]{Lemma}
\newtheorem{corollary}[theorem]{Corollary}

\newtheorem{definition}[theorem]{Definition}

\newtheorem{rem}{Remark}
\theoremstyle{remark}

\usepackage{xcolor}

\def\begeq{\begin{equation}}
\def\endeq{\end{equation}}

\allowdisplaybreaks

\title[Gradient estimates for $\Delta_pv+a(v+b)^q=0$]{Liouville theorems and new gradient estimates for positive solutions to $\Delta_pv+a(v+b)^q=0$ on a complete manifold}

\author{Youde Wang$^\S$}
\thanks{$\S\,$
1. School of Mathematics and Information Science, Guangzhou University, Guangzhou 510006, P. R. China; 2. Hua Loo-Keng Key Laboratory Mathematics, Institute of Mathematics, Academy of Mathematics and Systems Science, Chinese Academy of Sciences, Beijing 100190, China; 3. School of Mathematical Sciences, University of Chinese Academy of Sciences, Beijing 100049, China.
Email: wyd@math.ac.cn}

\author{Liqin Zhang$^\dag$*}
\thanks{$\ddag\,$
School of Mathematics and Information Science,
Guangzhou University,
Guangzhou 510006, P. R. China.
Email: 1061837643@qq.com\\
*Corresponding author}

\begin{document}
\maketitle
\rm\begin{abstract}\vspace{-2.5em}
{In this paper, we use the Saloff-Coste Sobolev inequality and Nash-Moser iteration method to study the local and global behaviors of positive solutions to the nonlinear elliptic equation $\Delta_pv+a(v+b)^q=0$ defined on a complete Riemannian manifold $\left(M,g\right)$ with Ricci lower bound, where $p>1$ is a constant and $\Delta_pv=\mathrm{div}\left(\left|\nabla v\right|^{p-2}\nabla v\right)$ is the usual $p$-Laplace operator. Under certain assumptions on $a$, $p$ and $q$, we derive some gradient estimates and Liouville type theorems for positive solutions to the above equation. In particular, under certain assumptions on $a$, $b$, $p$ and $q$ we show whether or not the exact Cheng-Yau $\log$-gradient estimates for the positive solutions to $\Delta_pv+av^q=0$ on $\left(M,g\right)$ with Ricci lower bound hold true is equivalent to whether or not the positive solutions to this equation fulfill Harnack inequality, and hence some new Cheng-Yau $\log$-gradient estimates are established.}\\

\noindent{\emph{Key words: gradient estimate; Nash-Moser iteration; Liouville type theorem}}
\end{abstract}

\textbf{\tableofcontents}

\section{\textbf{Introduction}}
Gradient estimate is a fundamental technique in the study of partial differential equations on a Riemannian manifold. Indeed, one can use gradient estimate to deduce Liouville type theorems (\cite{1,2,3,4,5,6}), to derive Harnack inequalities (\cite{3, 4}), to infer local and global behavior for solutions, to study the geometry of manifolds (\cite{CM1, 7, 4, 9}), etc.

On the other hand, it is well-known that the Liouville theorem has had a huge impact across many fields, such as complex analysis, partial differential equations, geometry, probability, discrete mathematics and complex and algebraic geometry. The impact of the Liouville theorem has been even larger as the starting point of many further developments. For more details on the Liouville properties of harmonic functions and some related theory of function on a manifold we refer to an expository paper \cite{CM} written by T. H. Colding (see also \cite{CM1}).

In this paper, we are concerned with the following equation
\begin{align}\label{1}
\Delta_pv+a(v+b)^q=0
\end{align}
defined on a complete Riemannian manifold $\left(M,g\right)$ equipped with a metric $g$, where $p>1$, $b\ge0$, $a$, $q$ are constants and $$\Delta_pv=\mathrm{div}\left(\left|\nabla v\right|^{p-2}\nabla v\right)$$ is the usual $p$-Laplace operator. Besides its own interest, such nonlinearites associated with \eqref{1} are also important as they arise from a class of quasilinear problems with quadratic growth in the gradient. Indeed, it is well-known (cf. \cite{KK}) that the change of variables
$$ v =e^u - b$$
reduces the quasilinear equation
$$ \Delta u + |\nabla u|^2 = F(x,u)$$
to the semilinear equation
$$ \Delta v=(v+b)F(x,\ln(v+b)),$$
where $b$ is a constant. By picking
$$F(x,u)\equiv e^{(q-1)u}$$
in the above equation, we can see that this equation is just \eqref{1} with $p=2$. Also, this indicates that it is of interest to study the following equation
$$ \Delta_p v=(v+b)F(x,\ln(v+b)),$$
where $1<p<\infty$. In fact, the first named author of this paper and coauthors \cite{Wang-Z-Z} have ever studied a special case of this equation, i.e., the case $p=2$.
\medskip

For the sake of simplicity, in this paper we will focus on the equation \eqref{1} and try to establish some new gradient estimates on the positive solutions to this equation. Now we recall some relative results in the previous literature with the equation.

When $b=0$, the equation \eqref{1} reduces to the following
$$\Delta_pv+av^q=0.$$
In the case $M$ is an Euclidean space, this equation was studied by Serrin and Zou in \cite{James} and some Liouville theorems and universal estimates were established. Very recently, He, the first named author of this article and Wei \cite{HWW} adopted a new way to employ the Nash-Moser iteration to study the gradient estimates of this equation on a complete Riemannian manifold.

In particular, it is worthy to point out that for the case $p=2$ the following new estimate was obtained in \cite{PWW1}
\begin{align}\label{PWW1}
\begin{split}
\frac{\left\vert \nabla u \right\vert^{2}}{u^{2}}+au^{q-1}
\leq~& \frac{2n}{2-n\max\{0,q-1\}}\left(\frac{C_{1}^{2}(n-1)(1+\sqrt{\kappa}R)+C_{2}}{R^{2}}\right.\\
&\left.+ 2\kappa+\frac{2nC_{1}^{2}}{(2+n\max\{0,q-1\})R^{2}}\right),
\end{split}
\end{align}
if the Ricci curvature of domain manifold satisfies $\mathrm{Ric}_g\geq-(n-1)\kappa$ and $q<\frac{n+2}{n}$. Obviously, this is a stronger estimate than the logarithmic gradient estimate (also see \cite{HW}). Wang-Wei \cite{6} also derived Cheng-Yau type gradient estimates for positive solutions to $\Delta u + u^q=0$ under the assumption
$$q\in \left(-\infty,~\frac{n+1}{n-1}+\frac{2}{\sqrt{n(n-1)}}\right).$$
Shortly after, He, the first named author of this article and Wei \cite{HWW} extended the Cheng-Yau estimate to the range
$$q\in \left(-\infty,~\frac{n+3}{n-1}\right).$$
Very recently, Z. Lu extended the estimate \eqref{PWW1} in \cite{PWW1} to the value range $q\in \left(-\infty,~\frac{n+3}{n-1}\right)$.
\medskip

The first goal of this paper is to give the gradient estimates of positive solutions with positive lower bounds to $\Delta_pv+av^q=0$, which is different from the exact $\log$-gradient estimate. This is equivalent to studying the positive solutions to \eqref{1}.

As the second goal of this paper we also try to answer the following questions:

\begin{itemize}
\item {\it A natural problem is whether or not the above $\frac{n+3}{n-1}(p-1)$ is optimal for one to derive the exact Cheng-Yau estimates for a positive $C^1$ solution to $\Delta_pu+ u^q=0$ on a complete manifold with Ricci curvature bounded below?}

\item {\it One would like to know whether or not the exact Cheng-Yau estimate holds true if $u$ is a $C^1$ smooth positive solution, which satisfies the standard Harnack inequality, to $\Delta_pu+ u^q=0$?}
\end{itemize}

Inspired by \cite{HHW, HWW, WZq, WZq1}, in the present paper we shall use the Nash-Moser iteration method to study the gradient estimate and the Liouville property of the above equation \eqref{1}, defined on a complete Riemannian manifold.

\subsection{Main results}\

By a solution $v$ of \eqref{1} in an (arbitrary) domain $\Omega$, we mean a positive solution $v\in C^1(\Omega)\cap C^3(\widetilde{\Omega})$, where $\widetilde{\Omega}=\left\{x\in\Omega|~|\nabla v(x)|\neq 0 \right\}$. It is well-known that any solution of \eqref{1} satisfies $v\in C^{1,\alpha}(\Omega)$ for some $\alpha\in(0,~1)$ (for example, see \cite{16,17,18}). Moreover, $v$ is in fact smooth in $\widetilde{\Omega}$.

For the sake of simplicity, we define
\begin{align*}
h=\beta(p-1)\left[\frac{p-n}{(n-1)^2}\beta+\frac{2}{n-1}\right]
\end{align*}
and
\begin{align*}
\begin{split}
\phi_\beta=
\begin{cases}
\mathop{\sup}\limits_{B(x_0,R)}v+b,&0<\beta<2, \\[3mm]
1,&\beta=2,\\[3mm]
\mathop{\inf}\limits_{B(x_0,R)}v+b,&\beta>2.
\end{cases}
\end{split}
\end{align*}
Furthermore, we suppose that $\beta$ satisfies the following condition
\begin{align}\label{beta}
\begin{split}
\begin{cases}
\beta\in\left(0,~\frac{2(n-1)}{n-p}\right),&1<p<n,\\[3mm]
\beta\in(0,~+\infty),&p\ge n.
\end{cases}
\end{split}
\end{align}
Now, we state our main results.
\begin{theorem}\label{theorem1.1}
Let $p>1$, $b\ge0$ and $(M,g)$ be an $n$-dim $(n\ge3)$ complete manifold with $\mathrm{Ric}\ge-(n-1)\kappa$, where $\kappa$ is a non-negative constant. Assume $v$ is a positive solution to equation \eqref{1} on the geodesic ball $B(x_0,2R)\subset M$. If the constants $a$, $p$ and $q$ satisfy one of the following two conditions
\begin{align*}
\frac{\beta}{2}\cdot\frac{n+1}{n-1}(p-1)-h^{\frac{1}{2}}<q<\frac{\beta}{2}\cdot\frac{n+1}{n-1}(p-1)+h^{\frac{1}{2}}\quad (a\neq 0);
\end{align*}
and
\begin{align*}
a\left[\frac{\beta}{2}\cdot\frac{n+1}{n-1}(p-1)-q\right]\ge0,
\end{align*}
where $\beta$ is a constant satisfying \eqref{beta},
then there exists a positive constant ${\mathcal{C}}={\mathcal{C}}(n,p,q, \beta)$ such that the following estimate holds true
\begin{align}\label{forest}
\mathop{\sup}\limits_{B\left(x_0,\frac{R}{2}\right)}\frac{|\nabla v|^2}{(v+b)^\beta}\le\mathcal{C}\frac{1+\kappa R^2}{R^2}\phi_\beta^{2-\beta}.
\end{align}
\end{theorem}

Here we would like to give a remark on the above theorem.
\begin{rem}
It is easy to see that, if $\beta=2$, there holds
$$\frac{\beta}{2}\cdot\frac{n+1}{n-1}(p-1)+h^{\frac{1}{2}}=\frac{n+3}{n-1}(p-1)$$
and
$$\frac{\beta}{2}\cdot\frac{n+1}{n-1}(p-1)-h^{\frac{1}{2}}=(p-1).$$
So, in the case $\beta=2$ and $b=0$, we can recover the conclusion of Theorem 1.1 in \cite{HWW}. On the other hand, from \eqref{forest} we can infer that there holds true
$$\mathop{\sup}\limits_{B\left(x_0,\frac{R}{2}\right)}|\nabla v|^2\le\mathcal{C}\frac{1+\kappa R^2}{R^2}\phi_\beta^{2}
= \mathcal{C}\frac{1+\kappa R^2}{R^2}(\mathop{\sup}\limits_{B(x_0,R)}v+b)^2,$$
if $p$ and $q$ satisfy the assumptions in the above theorem, where $\beta\in(0,\, 2)$.
\end{rem}
\medskip

For the sake of convenience, we define
\begin{align*}
\Psi(I)=\mathop{\sup}\limits_{\beta\in I}\left[\frac{\beta}{2}\cdot\frac{n+1}{n-1}(p-1)+h^{\frac{1}{2}}\right]
\end{align*}
and
\begin{align*}
    \Gamma(I)=\mathop{\inf}\limits_{\beta\in I}\left[\frac{\beta}{2}\cdot\frac{n+1}{n-1}(p-1)-h^{\frac{1}{2}}\right].
\end{align*}
If $2\in I$, obviously we have
$$\Psi(I)\geq \frac{n+3}{n-1}(p-1)\quad\quad\mbox{and}\quad\quad \Gamma(I)\leq (p-1).$$
So, we always have
$$\Psi\left(\left(0,~\frac{2(n-1)}{n-p}\right)\right)\geq \frac{n+3}{n-1}(p-1),$$
since $p>1$ and hence $2\in \left(0,~\frac{2(n-1)}{n-p}\right)$.
\medskip

By using the above notations and Theorem \ref{theorem1.1}, we can achieve the following conclusions.
\begin{corollary}\label{corollary1.2}
Let $p>1$, $b\ge0$ and $(M,g)$ be an $n$-dim $(n\ge3)$ complete manifold with $\mathrm{Ric}\ge-(n-1)\kappa$, where $\kappa$ is a non-negative constant. Assume $v$ is a positive solution to equation \eqref{1} on the geodesic ball $B(x_0,2R)\subset M$. If the constants $a$, $p$ and $q$ satisfy one of the following five conditions
\begin{itemize}
\item $a>0$, $p\ge n$ and $q\in \mathbb{R}$;
\item $a>0$, $1<p<n$ and $q<\Psi\left(\left(0,~\frac{2(n-1)}{n-p}\right)\right)$;
\item $a<0$, $p\ge n$ and $q>\Gamma((0,~+\infty))$;
\item $a<0$, $1<p<n$ and $q>\Gamma\left(\left(0,~\frac{2(n-1)}{n-p}\right)\right)$;
\item $a=0$, $p>1$,
\end{itemize}
then there exist positive constants ${\mathcal{C}}={\mathcal{C}}(n,p,q)$ and $\beta=\beta(n,p,q)\in(0,~+\infty)$, such that the following estimate holds true
\begin{align*}
\mathop{\sup}\limits_{B\left(x_0,\frac{R}{2}\right)}\frac{|\nabla v|^2}{(v+b)^\beta}\le\mathcal{C}\frac{1+\kappa R^2}{R^2}\phi_\beta^{2-\beta}.
\end{align*}
\end{corollary}

\begin{rem}
Here we would like to give several comments on the above corollary.
\begin{itemize}
\item In the case 2, i.e., $a>0$ and $1<p<n$, if there exists a point $\beta_0\in\left(0,~\frac{2(n-1)}{n-p}\right)$ such that
$$\Psi\left(\left(0,~\frac{2(n-1)}{n-p}\right)\right)=\left[\frac{\beta}{2}\cdot\frac{n+1}{n-1}(p-1)+h^{\frac{1}{2}}\right]\bigg|_{\beta=\beta_0},$$ then the condition $q<\Psi\left(\left(0,~\frac{2(n-1)}{n-p}\right)\right)$ can be relaxed to $q\le\Psi\left(\left(0,~\frac{2(n-1)}{n-p}\right)\right)$. For another four cases, we can obtain similar conclusions.

\item If $1<p<n$, then it is not difficult to see that there holds true
$$\frac{n+1}{n-p}(p-1)\le\Psi\left(\left(0,~\frac{2(n-1)}{n-p}\right)\right).$$
Usually, the above is a strict inequality, for instance, if we let $n=3$ and $p=2$, then
\begin{align}\label{example}
\frac{\beta}{2}\cdot\frac{n+1}{n-1}(p-1)+h^{\frac{1}{2}}=\beta+\sqrt{\beta\left(1-\frac{\beta}{4}\right)}
\end{align}
and $\beta\in (0,~4)$. Hence, we can easily check that \eqref{example} can attain its maximum at an interior point $\beta_0=2+\frac{4\sqrt{5}}{5}\in(0,4)$. Therefore, we can know that
\begin{align*}
\Psi((0,4))=2+\sqrt{5}>\frac{n+1}{n-p}(p-1)=4.
\end{align*}
On the other hand, we also have
$$\Psi((0,4))=2+\sqrt{5}>\frac{n+3}{n-1}(p-1)=3.$$
This indicates that for $q\geq \frac{n+3}{n-1}(p-1)$ one also derives the gradient estimate.
\end{itemize}
\end{rem}
\medskip

\begin{corollary}\label{corollary1.3}
Let $p>1$, $b\ge0$ and $(M,g)$ be an $n$-dim $(n\ge3)$ complete manifold with $\mathrm{Ric}\ge-(n-1)\kappa$, where $\kappa$ is a non-negative constant. Assume $v$ is a positive solution to equation \eqref{1} on the geodesic ball $B(x_0,2R)\subset M$. If the constants $a$, $p$ and $q$ satisfy the following conditions
\begin{align*}
a<0~\quad\mbox{and}~\quad q>\Gamma((0,2]),
\end{align*}
then there exist positive constants ${\mathcal{C}}={\mathcal{C}}(n,p,q)$ and $\beta=\beta(n,p,q)\in(0,2]$, such that the following estimate holds true
\begin{align*}
\mathop{\sup}\limits_{B\left(x_0,\frac{R}{2}\right)}\frac{|\nabla v|^2}{(v+b)^\beta}\le\mathcal{C}\frac{1+\kappa R^2}{R^2}\bigg(\mathop{\sup}\limits_{B(x_0,R)}v+b\bigg)^{2-\beta}.
\end{align*}
\end{corollary}

If $b>0$, by using Theorem \ref{theorem1.1} we can achieve the following conclusion.
\begin{corollary}\label{corollary1.4}
Let $p>1$, $b>0$ and $(M,g)$ be an $n$-dim $(n\ge3)$ complete manifold with $\mathrm{Ric}\ge-(n-1)\kappa$, where $\kappa$ is a non-negative constant. Assume $v$ is a positive solution to equation \eqref{1} on the geodesic ball $B(x_0,2R)\subset M$. If the constants $a$, $p$ and $q$ satisfy one of the following five conditions
\begin{itemize}
\item $a>0$, $p\ge n$ and $q\in \mathbb{R}$;
\item $a>0$, $1<p<n$ and $q<\Psi\left(\left[2,~\frac{2(n-1)}{n-p}\right)\right)$;
\item $a<0$, $p\ge n$ and $q>\Gamma([2,~+\infty))$;
\item $a<0$, $1<p<n$ and $q>\Gamma\left(\left[2,~\frac{2(n-1)}{n-p}\right)\right)$;
\item $a=0$, $p>1$,
\end{itemize}
then there exist positive constants ${\mathcal{C}}={\mathcal{C}}(n,p,q,b)$ and $\beta=\beta(n,p,q)\in[2,~+\infty)$, such that the following estimate holds true
\begin{align*}
\mathop{\sup}\limits_{B\left(x_0,\frac{R}{2}\right)}\frac{|\nabla v|^2}{(v+b)^\beta}\le\mathcal{C}\frac{1+\kappa R^2}{R^2}.
\end{align*}
\end{corollary}
\medskip

In the case $b=0$, then \eqref{1} reduces to
\begin{align}\label{b=0}
\Delta_pv+av^q=0.
\end{align}
In the following theorem we show whether or not the exact Cheng-Yau $\log$-gradient estimates for the positive solutions to this equation \eqref{b=0} holds true are equivalent to whether or not the positive solutions to \eqref{b=0} fulfill Harnack inequality.

\begin{theorem}\label{corollary1.5}
Let $p>1$ and $(M,g)$ be an $n$-dim $(n\ge3)$ complete manifold with $\mathrm{Ric}\ge-(n-1)\kappa$, where $\kappa$ is a non-negative constant. Assume $v$ is a positive solution to equation \eqref{b=0} on the geodesic ball $B(x_0,2R)\subset M$ and satisfies the following Harnack inequality
\begin{align*}
\mathop{\sup}\limits_{B\left(x_0,R\right)}v\le l\mathop{\inf}\limits_{B\left(x_0,R\right)}v.
\end{align*}
If the constants $a$, $q$ and $p$ satisfy one of the following five conditions
\begin{itemize}
\item $a>0$, $p\ge n$ and $q\in \mathbb{R}$;
\item $a>0$, $1<p<n$ and $q<\Psi\left(\left(0,~\frac{2(n-1)}{n-p}\right)\right)$;
\item $a<0$, $p\ge n$ and $q>\Gamma((0,~+\infty))$;
\item $a<0$, $1<p<n$ and $q>\Gamma\left(\left(0,~\frac{2(n-1)}{n-p}\right)\right)$;
\item $a=0$, $p>1$,
\end{itemize}
then there exists a positive constant ${\mathcal{C}}={\mathcal{C}}(n,p,q,l)$ such that the following estimate holds true
\begin{align*}
\mathop{\sup}\limits_{B\left(x_0,\frac{R}{2}\right)}\frac{|\nabla v|^2}{v^2}\le\mathcal{C}\frac{1+\kappa R^2}{R^2}.
\end{align*}
\end{theorem}

As a direct consequence, we obtain

\begin{corollary}\label{corollary1.5*}
Let $p>1$ and $(M,g)$ be an $n$-dim $(n\ge3)$ complete manifold with nonnegative Ricci curvature. Assume $v$ is a positive solution to equation \eqref{b=0} on any given geodesic ball $B(x_0,2R)\subset M$ and satisfies the following Harnack inequality
\begin{align*}
\mathop{\sup}\limits_{B\left(x_0,R\right)}v\le l\mathop{\inf}\limits_{B\left(x_0,R\right)}v,
\end{align*}
where $l$ is independent of $v$ and $R$. If the constants $a$, $q$ and $p$ satisfy one of the following five conditions
\begin{itemize}
\item $a>0$, $p\ge n$ and $q\in \mathbb{R}$;
\item $a>0$, $1<p<n$ and $q<\Psi\left(\left(0,~\frac{2(n-1)}{n-p}\right)\right)$;
\item $a<0$, $p\ge n$ and $q>\Gamma((0,~+\infty))$;
\item $a<0$, $1<p<n$ and $q>\Gamma\left(\left(0,~\frac{2(n-1)}{n-p}\right)\right)$;
\item $a=0$, $p>1$,
\end{itemize}
then there exist a positive constant ${\mathcal{C}}={\mathcal{C}}(n,p,q,l)$ such that the following estimate holds true
\begin{align*}
\mathop{\sup}\limits_{B\left(x_0,\frac{R}{2}\right)}\frac{|\nabla v|^2}{v^2}\le\mathcal{C}\cdot\frac{1}{R^2}.
\end{align*}

Conversely, if the above $\log$-gradient estimate holds true, then, for any given $B\left(x_0,R\right)\subset M$ there holds
$$\mathop{\sup}\limits_{B\left(x_0,R\right)}v\le l\mathop{\inf}\limits_{B\left(x_0,R\right)}v$$
where $l$ is independent of $v$ and $R$.
\end{corollary}

If we consider \eqref{b=0} in $\mathbb{R}^n$ ($n\ge3$), we can achieve the following conclusion.
\begin{corollary}\label{corollary1.6}
Assume $v$ is a positive solution to equation \eqref{b=0} on the ball $B(x_0,2R)\subset \mathbb{R}^n$.  If the constants $a$, $q$ and $p$ satisfy one of the following four conditions
\begin{itemize}
\item $a>0$, $1<p<n$, $p\neq q$ and $q\in\left(p-1,~\frac{(p-1)n}{n-p}\right)$;
\item $a>0$, $p\ge n$, $q\neq p$ and $q\in\left(0,~+\infty\right)$;
\item $a\ge 1$ and $1<p=q<n<p^2$;
\item $a\ge 1$ and $p=q\ge n$,
\end{itemize}
then there exist a positive constant ${\mathcal{C}}={\mathcal{C}}(n,p,q,a)$ such that the following estimate holds true
\begin{align*}
\mathop{\sup}\limits_{B\left(x_0,\frac{R}{2}\right)}\frac{|\nabla v|^2}{v^2}\le\frac{\mathcal{C}}{R^2}.
\end{align*}
\end{corollary}

\begin{rem}
The corollary tells us that $\frac{n+3}{n-1}(p-1)$ is not optimal bound to derive the exact Cheng-Yau type $\log$-gradient estimate, since there holds true
$$\mathop{\sup}\limits_{B\left(x_0,\frac{R}{2}\right)}\frac{|\nabla v|^2}{v^2}\le\frac{\mathcal{C}}{R^2}$$
in the case $a>0$, $p\ge n$, $q\neq p$ and $q\in\left(0,~+\infty\right)$.
\end{rem}

By using Corollary \ref{corollary1.4}, we can obtain the following Liouville type theorems.
\begin{theorem}\label{theorem1.7}
Let $(M,g)$ be a complete non-compact Riemannian manifold with non-negative Ricci curvature and $\dim M=n\ge3$. Furthermore, let $v$ be a positive solution to equation \eqref{1} on $M$. If $b>0$ and the constants $a$, $p$ and $q$ satisfy one of the following four conditions
\begin{itemize}
\item $a>0$, $p\ge n$ and $q\in \mathbb{R}$;
\item $a>0$, $1<p<n$ and $q<\Psi\left(\left[2,~\frac{2(n-1)}{n-p}\right)\right)$;
\item $a<0$, $p\ge n$ and $q>\Gamma([2,~+\infty))$;
\item $a<0$, $1<p<n$ and $q>\Gamma\left(\left[2,~\frac{2(n-1)}{n-p}\right)\right)$,
\end{itemize}
then \eqref{1} admits no positive solution.
\end{theorem}

As a direct consequence of Corollary \ref{corollary1.6}, we have
\begin{theorem}\label{theorem1.8}
Assume $v$ is a positive solution to equation \eqref{b=0} on $\mathbb{R}^n$ and $n\ge3$. If the constants $a$, $p$ and $q$ satisfy one of the following four conditions
\begin{itemize}
\item $a>0$, $1<p<n$, $p\neq q$ and $q\in\left(p-1,~\frac{(p-1)n}{n-p}\right)$;
\item $a>0$, $p\ge n$, $p\neq q$ and $q\in\left(0,~+\infty\right)$;
\item $a\ge 1$ and $1<p=q<n<p^2$;
\item $a\ge 1$ and $p=q\ge n$,
\end{itemize}
then \eqref{b=0} admits no positive solution.
\end{theorem}

In the above conclusions, we always suppose that $\dim M=n\ge 3$. In fact, for the case $\dim M=2$ we can also obtain similar conclusions.
Since its proof is similar to the case $\dim M\ge 3$, we will not give the discussions for this case.

\subsection{Main ideas of proof and the organization of paper}\

In order to give the gradient estimates, we consider the linearized operator of $\mathcal{L}_{p}$ of $p$-Laplace operator at a solution $v$, and  let $\mathcal{L}_{p}$ act on an auxiliary function given by
$$F(v)=\frac{|\nabla v|^2}{(v+b)^\beta}, \quad \beta >0.$$
To choose such an auxiliary function is motivated by the gradient estimates established in \cite{WjW} when they considered another equation related to Ricci solitons. Then, we need to establish some suitable point-wise estimate of $\mathcal{L}_{p}(F)$ by the method or the techniques due to Cheng- Yau \cite{Ch-Y} and Yau \cite{3}, so that we can take a Nash-Moser iteration scheme to give the $L^\infty$-norm of $F(v)$. The Saloff-Coste's Sobolev inequalities play an important role in our arguments.
\medskip

Our paper is organized as follows: In Section 2, we first recall some preliminary and then establish some important lemmas, which will play a key role in the Nash-Moser iteration process. Next in Section 3, we prove some important gradient estimates, which is the main body of this paper. In Section 4, we give some necessary proofs of the main results.
\medskip

\section{\textbf{Preliminaries}}
Throughout this paper, we let $(M,g)$ be an $n$-dim Riemannian manifold $(n\ge3)$, and $\nabla$ denote the corresponding Levi-Civita connection to the metric $g$. We denote the volume form on $(M, g)$ by
$$d\mathrm{vol}=\sqrt{\det(g_{ij})}dx_1\wedge\cdots\wedge dx_n,$$
where $(x_1,\cdots,x_n)$ is a local coordinate chart, and for simplicity we usually omit the volume form of integral over $M$.

\begin{definition}
We say that $v\in C^1(M)\cap W_{loc}^{1,p}(M)$ is a weak solution of \eqref{1} if for all $\psi\in W_{0}^{1,p}(M)$ we have
\begin{align*}
\int_M{\left|\nabla v\right|^{p-2}\langle\nabla v,\nabla\psi\rangle}=\int_M{a(v+b)^q}\psi.
\end{align*}
\end{definition}

Next, we recall the Saloff-Coste's Sobolev inequalities (see \cite{14}), which shall play a key role in our proof of the main theorems.
\begin{lemma}(Saloff-Coste \cite{14})
Let $(M,g)$ be a complete manifold with $\mathrm{Ric}\ge-(n-1)\kappa$. For $n>2$, there exists a positive constant $C_n$ depending only on $n$, such that for all $B\subset M$ of radius $R$ and volume $V$ we have for $h_1\in C_0^\infty(B)$
\begin{align*}
\left\|h_1\right\|^2_{L^{\frac{2n}{n-2}}(B)}\le \exp\left\{C_n(1+\sqrt{\kappa}R)\right\}V^{-\frac{2}{n}}R^2\left(\int_B{\left|\nabla h_1\right|^2+R^{-2}h_1^2}\right).
\end{align*}
For $n=2$, the above inequality holds with n replaced by any fixed $n'>2$.
\end{lemma}

Now we consider the linearization operator $\mathcal{L}_p$ of $p$-Laplace operator:
\begin{align}\label{2.1}
\mathcal{L}_p(\psi)=\mathrm{div}\left[f^{\frac{p}{2}-1}A\left(\nabla\psi\right)\right],
\end{align}
where $u=v+b$, $f=\left|\nabla u\right|^2$ and
\begin{align}\label{2.2}
A\left(\nabla\psi\right)=\nabla\psi+(p-2)f^{-1}\langle\nabla\psi,\nabla u\rangle\nabla u.
\end{align}
We first derive an useful expression of $\mathcal{L}_p(f)$.
\begin{lemma}\label{2.3}
The equality
\begin{align}
\mathcal{L}_p(f)=\left(\frac{p}{2}-1\right)f^{\frac{p}{2}-2}\left|\nabla f\right|^2+2f^{\frac{p}{2}-1}\left(\left|\nabla\nabla u\right|^2+\mathrm{Ric}\left(\nabla u,\nabla u\right)\right)+2\langle\nabla\Delta_pu,\nabla u\rangle
\end{align}
holds point-wisely in $\{x:~f(x)>0\}$.
\end{lemma}

\begin{proof}
By the definition of $A$ in \eqref{2.2}, we have
\begin{align}\label{2.4}
 A\left(\nabla f\right)=\nabla f+(p-2)f^{-1}\langle\nabla f,\nabla u\rangle\nabla u.
\end{align}
Combining \eqref{2.1} and \eqref{2.4} together, we obtain
\begin{align}
\begin{split}\label{2.5}
\mathcal{L}_p(f)=&\left(\dfrac{p}{2}-1\right)f^{\frac{p}{2}-2}|\nabla f|^2+f^{\frac{p}{2}-1}\Delta f+(p-2)\left(\dfrac{p}{2}-2\right)f^{\frac{p}{2}-3}\langle\nabla f,\nabla u\rangle^2\\
&+(p-2)f^{\frac{p}{2}-2}\langle\nabla\langle\nabla f,\nabla u\rangle,\nabla u\rangle+(p-2)f^{\frac{p}{2}-2}\langle\nabla f,\nabla u\rangle\Delta u.
\end{split}
\end{align}
On the other hand, by the definition of the $p$-Laplacian, we have
\begin{align}
\begin{split}\label{2.6}
2\langle\nabla \Delta_pu,\nabla u\rangle=&(p-2)\left(\dfrac{p}{2}-2\right)f^{\frac{p}{2}-3}\langle\nabla f,\nabla u\rangle^2+(p-2)f^{\frac{p}{2}-2}\langle\nabla\langle\nabla f,\nabla u\rangle,\nabla u\rangle\\
&+(p-2)f^{\frac{p}{2}-2}\langle\nabla f,\nabla u\rangle\Delta u+2f^{\frac{p}{2}-1}\langle\nabla \Delta u,\nabla u\rangle.
\end{split}
\end{align}
Hence, we combine \eqref{2.5} and \eqref{2.6} together to obtain
\begin{align}
\begin{split}\label{2.7}
\mathcal{L}_p(f)=\left(\dfrac{p}{2}-1\right)f^{\frac{p}{2}-2}|\nabla f|^2+f^{\frac{p}{2}-1}\Delta f+2\langle\nabla \Delta_pu,\nabla u\rangle-2f^{\frac{p}{2}-1}\langle\nabla \Delta u,\nabla u\rangle.
\end{split}
\end{align}
Furthermore, by \eqref{2.7} and the following Bochner formula
\begin{center}
$
\frac{1}{2}\Delta f=\left|\nabla\nabla u\right|^2+\mathrm{Ric}\left(\nabla u,\nabla u\right)+\langle\nabla \Delta u,\nabla u\rangle,
$
\end{center}
we get
\begin{center}
$
\mathcal{L}_p(f)=\left(\dfrac{p}{2}-1\right)f^{\frac{p}{2}-2}\left|\nabla f\right|^2+2f^{\frac{p}{2}-1}\left(\left|\nabla\nabla u\right|^2+\mathrm{Ric}\left(\nabla u,\nabla u\right)\right)+2\langle\nabla\Delta_pu,\nabla u\rangle.
$
\end{center}
Thus, we finish the proof.
\end{proof}

In the last section, we are going to use the following lemmas.

\begin{lemma}\label{lemma2.4}(\cite{James}, Theorem 4.1 ($a$))
Suppose $n>m$ and $s\in\left(m,~\frac{m(n-1)}{n-m}\right)$. Let $\omega$ be a non-negative weak solution of the differential inequality
\begin{align}\label{2.8}
\omega^{s-1}\le-\Delta_m\omega\le\Lambda \omega^{s-1}\quad in~\Omega,
\end{align}
for some constant $\Lambda>1$. ($\Omega$ is a domain in $\mathbb{R}^n$) Then there is a constant $\mathcal{C}=\mathcal{C}(n,m,s,\Lambda)>0$ such that
\begin{align}\label{2.9}
\mathop{\sup}\limits_{B_R}\omega(x)\le\mathcal{C}\mathop{\inf}\limits_{B_R}\omega(x).
\end{align}
($B_R$ denotes a ball with radius $R$.)
\end{lemma}

\begin{lemma}\label{lemma2.5}(\cite{James}, Theorem 4.3 ($a$)) Let $n\le m$. Then: Assume the hypotheses of Lemma \ref{lemma2.4}, except that the condition $s\in\left(m,~\frac{m(n-1)}{n-m}\right)$ is replaced by $s\in(1,+\infty)$, that is, $\frac{m(n-1)}{n-m}=+\infty$. Then \eqref{2.9} is valid with $\mathcal{C}=\mathcal{C}(n,m,s,\Lambda)>0$.
\end{lemma}
\medskip

\section{\textbf{Gradient estimates}}

\subsection{\textbf{Estimate for the linearized operator of $p$-Laplace}}\

First, we need to give the pointwise estimate of $\mathcal{L}_{p}(F)$, where $$F=\dfrac{f}{(v+b)^\beta},~\quad (\beta>0)$$ and $\mathcal{L}_{p}$ is the linearized operator of $p$-Laplacian at $u=v+b$.
\begin{lemma}\label{lemma3.1}The equality
\begin{align}
\begin{split}
    \mathcal{L}_p(F)=&u^{-\beta}\mathcal{L}_p(f)+\beta (\beta+1)(p-1)u^{-\beta-2}f^{\frac{p}{2}+1}-\beta \left(1+\frac{p}{2}\right)(p-1)u^{-\beta-1}f^{\frac{p}{2}-1}\langle\nabla f,\nabla u\rangle\\
    &-\beta (p-1)u^{-\beta-1}f^\frac{p}{2}\Delta u
\end{split}
\end{align}
holds point-wisely in $\{ x:~f(x)>0$ $\}$.
\end{lemma}
\begin{proof}
By the definition of $A$ in \eqref{2.2}, we have
 \begin{align}
     \label{3.2}&A(\nabla F)=u^{-\beta}A(\nabla f)-\beta u^{-\beta-1}fA(\nabla u),\\
     \label{3.3}&A(\nabla u)=(p-1)\nabla u
     \end{align}
 and
 \begin{align}\label{3.4}
     A(\nabla f)=\nabla f+(p-2)f^{-1}\langle\nabla u,\nabla f\rangle\nabla u.
 \end{align}
 Combining \eqref{2.1} and \eqref{3.2} together, we obtain
 \begin{align}\label{3.5}
     \mathcal{L}_p(F)=\mathrm{div}\left[u^{-\beta}f^{\frac{p}{2}-1}A\left(\nabla f\right) \right]-\beta\mathrm{div}\left[u^{-\beta-1}f^{\frac{p}{2}}A\left(\nabla u\right) \right].
 \end{align}
 Direct computation shows that
 \begin{align}\label{3.6}
     \mathrm{div}\left[u^{-\beta}f^{\frac{p}{2}-1}A\left(\nabla f\right) \right]=-\beta u^{-\beta-1}f^{\frac{p}{2}-1}\langle A\left(\nabla f\right),\nabla u\rangle+u^{-\beta}\mathrm{div}\left[f^{\frac{p}{2}-1}A\left(\nabla f\right)\right]
 \end{align}
and
\begin{align}\label{3.7}
\begin{split}
\mathrm{div}\left[u^{-\beta-1}f^{\frac{p}{2}}A\left(\nabla u\right) \right]=&-(\beta+1)u^{-\beta-2}f^{\frac{p}{2}}\langle A\left(\nabla u\right),\nabla u\rangle+\frac{p}{2}u^{-\beta-1}f^{\frac{p}{2}-1}\langle A(\nabla u),\nabla f\rangle\\
&+u^{-\beta-1}f^{\frac{p}{2}}\mathrm{div}A\left(\nabla u\right).
\end{split}
\end{align}
By substituting \eqref{2.1} and \eqref{3.4} into \eqref{3.6}, we have
\begin{align}\label{3.8}
\mathrm{div}\left[u^{-\beta}f^{\frac{p}{2}-1}A\left(\nabla f\right) \right]=-\beta (p-1)u^{-\beta-1}f^{\frac{p}{2}-1}\langle\nabla f,\nabla u\rangle+u^{-\beta}\mathcal{L}_p(f).
\end{align}
Substituting \eqref{3.3} into \eqref{3.7} leads to
\begin{align}\label{3.9}
\begin{split}
\mathrm{div}\left[u^{-\beta-1}f^{\frac{p}{2}}A\left(\nabla u\right) \right]=&-(\beta+1)(p-1)u^{-\beta-2}f^{\frac{p}{2}+1}+\frac{p}{2}(p-1)u^{-\beta-1}f^{\frac{p}{2}-1}\langle \nabla f,\nabla u\rangle\\
&+(p-1)u^{-\beta-1}f^{\frac{p}{2}}\Delta u.
\end{split}
\end{align}
Now, we plug \eqref{3.8} and \eqref{3.9} into\eqref{3.5} to derive the required equality, and hence finish the proof of Lemma \ref{lemma3.1}.
\end{proof}

Using  Lemma \ref{2.3} and Lemma \ref{lemma3.1}, we can establish the following Lemma:
\begin{lemma}\label{lemma3.2} Let $v$ be a positive solution of equation \eqref{1} in $\Omega\subset M$. Then, the following
\begin{align}\label{3.10}
\begin{split}
\mathcal{L}_p(F)=&\left(\frac{p}{2}-1\right)u^\beta f^{\frac{p}{2}-2}|\nabla F|^2+2a\left[\frac{\beta}{2}(p-1)-q\right]u^{q-\beta-1}f-p\beta u^{-1}f^{\frac{p}{2}-1}\langle\nabla F,\nabla u\rangle\\
&+2u^{-\beta}f^{\frac{p}{2}-1}\left(\left|\nabla\nabla u\right|^2+\mathrm{Ric}\left(\nabla u,\nabla u\right)\right)+\left[-\frac{1}{2}p\beta^2+(p-1)\beta\right]u^{-\beta-2}f^{\frac{p}{2}+1}\\
\end{split}
\end{align}
holds point-wisely in $\{ x\in\Omega:~f(x)>0$ $\}$.
\end{lemma}

\begin{proof}
By summarizing Lemma \ref{2.3} and Lemma \ref{lemma3.1} we can achieve that
\begin{align}\label{3.11}
\begin{split}
\mathcal{L}_p(F)=&u^{-\beta}\left[\left(\frac{p}{2}-1\right)f^{\frac{p}{2}-2}\left|\nabla f\right|^2+2f^{\frac{p}{2}-1}\left(\left|\nabla\nabla u\right|^2+\mathrm{Ric}\left(\nabla u,\nabla u\right)\right)+2\langle\nabla\Delta_pu,\nabla u\rangle\right]\\
&+\beta (\beta+1)(p-1)u^{-\beta-2}f^{\frac{p}{2}+1}-\beta \left(1+\frac{p}{2}\right)(p-1)u^{-\beta-1}f^{\frac{p}{2}-1}\langle\nabla f,\nabla u\rangle\\
&-\beta (p-1)u^{-\beta-1}f^\frac{p}{2}\Delta u.
\end{split}
\end{align}
Since $F=\frac{f}{u^\beta}$, we can infer that
\begin{align*}
\nabla f=\beta u^{-1}f\nabla u+u^\beta\nabla F.
\end{align*}
Hence, we have
\begin{align}\label{3.12}
\langle\nabla f,\nabla u\rangle=\beta u^{-1}f^2+u^\beta\langle\nabla F,\nabla u\rangle
\end{align}
and
\begin{align}\label{3.13}
|\nabla f|^2=\beta^2u^{-2}f^3+2\beta u^{\beta-1}f\langle\nabla F,\nabla u\rangle+u^{2\beta}|\nabla F|^2.
\end{align}
Substituting $\Delta_p u+au^q=0$, \eqref{3.12} and \eqref{3.13} into \eqref{3.11}, we obtain
\begin{align}\label{3.14}
\begin{split}
\mathcal{L}_p(F)=&\left(\frac{p}{2}-1\right)u^\beta f^{\frac{p}{2}-2}|\nabla F|^2-2aqu^{q-\beta-1}f+2u^{-\beta}f^{\frac{p}{2}-1}\left(\left|\nabla\nabla u\right|^2+\mathrm{Ric}\left(\nabla u,\nabla u\right)\right)\\
&+\left[-\frac{1}{2}(p^2-2p+2)\beta^2+(p-1)\beta\right]u^{-\beta-2}f^{\frac{p}{2}+1}-\frac{1}{2}(p^2-p+2)\beta u^{-1}f^{\frac{p}{2}-1}\langle\nabla F,\nabla u\rangle\\
&-\beta(p-1)u^{-\beta-1}f^{\frac{p}{2}}\Delta u.
\end{split}
\end{align}
Since $\Delta_p u+au^q=0$ and
\begin{align*}
\Delta_pu=\mathrm{div}\left(f^{\frac{p}{2}-1}\nabla u\right)=\left(\frac{p}{2}-1\right)f^{\frac{p}{2}-2}\langle\nabla f,\nabla u\rangle+f^{\frac{p}{2}-1}\Delta u,
\end{align*}
it is easy to verify that
\begin{align}\label{3.15}
\Delta u=\left(1-\frac{p}{2}\right)f^{-1}\langle\nabla f,\nabla u\rangle-au^qf^{1-\frac{p}{2}}.
\end{align}
Substituting \eqref{3.12} into the above \eqref{3.15} yields
\begin{align}\label{3.16}
\Delta u=\left(1-\frac{p}{2}\right)u^\beta f^{-1}\langle\nabla F,\nabla u\rangle+\left(1-\frac{p}{2}\right)\beta u^{-1}f-au^qf^{1-\frac{p}{2}}.
\end{align}
Hence, we substitute \eqref{3.16} into \eqref{3.14} to derive the required equality. Thus we complete the proof of Lemma \ref{lemma3.2}.
\end{proof}

In order to establish the pointwise estimate of $\mathcal{L}_{p}(F^\alpha)$, we begin with the following equality about
$\mathcal{L}_p(F^\alpha)$.
\begin{lemma}\label{lemma3.3}
Let $\alpha>1$ and $v$ be a positive solution of equation \eqref{1} in $\Omega\subset M$. Then, the following
\begin{align}\label{3.17}
\begin{split}
\frac{1}{\alpha} F^{2-\alpha}\mathcal{L}_p(F^\alpha)=&\left(\alpha+\frac{p}{2}-2\right)f^{\frac{p}{2}-1}|\nabla F|^2+(\alpha-1)(p-2)f^{\frac{p}{2}-2}\langle\nabla F,\nabla u\rangle^2\\
&+2u^{-2\beta}f^{\frac{p}{2}}\left(\left|\nabla\nabla u\right|^2+\mathrm{Ric}\left(\nabla u,\nabla u\right)\right)-\beta pu^{-\beta-1}f^{\frac{p}{2}}\langle\nabla F,\nabla u\rangle\\
&+2a\left[\frac{\beta}{2}(p-1)-q\right]u^{q-2\beta-1}f^2+\left[-\frac{p}{2}\beta^2+(p-1)\beta\right]u^{-2\beta-2}f^{\frac{p}{2}+2}\\
\end{split}
\end{align}
holds point-wisely in $\{ x\in\Omega:~f(x)>0\}$.
\end{lemma}

\begin{proof}
By the definition of $A$ in \eqref{2.2}, we have
\begin{align}\label{3.18}
A\left(\nabla\left(F^\alpha\right)\right)=\alpha F^{\alpha-1}A(\nabla F)
\end{align}
and
\begin{align}\label{3.19}
\langle A(\nabla F),\nabla F\rangle=|\nabla F|^2+(p-2)f^{-1}\langle\nabla F,\nabla u\rangle^2.
\end{align}
Combining \eqref{2.1}, \eqref{3.18} and \eqref{3.19}, we obtain
\begin{align}\label{3.20}
\begin{split}
\mathcal{L}_p(F^\alpha)=&\alpha\mathrm{div}\left[F^{\alpha-1}f^{\frac{p}{2}-1}A(\nabla F)\right]\\
=&\alpha(\alpha-1)F^{\alpha-2}f^{\frac{p}{2}-1}\langle A(\nabla F),\nabla F\rangle+\alpha F^{\alpha-1}\mathcal{L}_p(F)\\
=&\alpha(\alpha-1)F^{\alpha-2}f^{\frac{p}{2}-1}\left[|\nabla F|^2+(p-2)f^{-1}\langle\nabla F,\nabla u\rangle^2\right]+\alpha F^{\alpha-1}\mathcal{L}_p(F).
\end{split}
\end{align}
In view of \eqref{3.10} and \eqref{3.20}, we can derive the required equality and complete the proof of Lemma \ref{lemma3.3}.
\end{proof}

Next, we need to consider the point-wise estimate of $\mathcal{L}_p(F^\alpha)$. Therefore, we begin with the following two Lemma.

\begin{lemma}\label{lemma3.4}
Let $a\neq 0$, $\alpha>1$ and $v$ be a positive solution of \eqref{1} in $\Omega\subset M$ with $\rm{Ric}\ge-(n-1)\kappa$. Set
\begin{align}\label{3.21}
H=\beta(p-1)\left[\frac{p-n}{2(n-1)}\beta+1\right]-\frac{\left[\frac{\beta}{2}\frac{n+1}{n-1}(p-1)-q\right]^2}{\frac{2}{n-1}-\frac{p-1}{(n-1)^2}\left[\alpha+\frac{p-n}{2(n-1)}\right]^{-1}}.
\end{align}
Then, the following
\begin{align}\label{3.22}
\frac{1}{\alpha} F^{2-\alpha}\mathcal{L}(F^\alpha)\ge-2(n-1)\kappa u^{-2\beta}f^{\frac{p}{2}+1}+(p-1)\beta\frac{p-n}{n-1}u^{-\beta-1}f^{\frac{p}{2}}\langle\nabla F,\nabla u\rangle+Hu^{-2\beta-2}f^{\frac{p}{2}+2}
\end{align}
holds point-wisely in $\{ x\in\Omega:~f(x)>0$ $\}$.
\end{lemma}

\begin{proof}
Let $\{e_1,\cdots,e_n\}$ be an orthonormal frame of $TM$ on a domain with $f\neq0$ such that $e_1=\dfrac{\nabla u}{|\nabla u|}$. We hence infer that
    \begin{align}\label{3.23}
        4\sum_{i=1}^nu_{1i}^2&=f^{-1}|\nabla f|^2,
    \end{align}
    \begin{align}\label{3.24}
        u_{11}=\frac{\langle\nabla f,\nabla u\rangle}{2f}
    \end{align}
    and
    \begin{align}\label{3.25}
        \Delta_pu=(p-1)f^{\frac{p}{2}-1}u_{11}+f^{\frac{p}{2}-1}\sum_{i=2}^nu_{ii}.
    \end{align}
Substituting \eqref{3.12} and \eqref{3.13} into \eqref{3.23} and \eqref{3.24} respectively leads to
\begin{align}\label{3.26}
4\sum_{i=1}^nu_{1i}^2=u^{2\beta}f^{-1}|\nabla F|^2+\beta^2u^{-2}f^2+2\beta u^{\beta-1}\langle\nabla F,\nabla u\rangle
\end{align}
and
\begin{align}\label{3.27}
2u_{11}=\beta u^{-1}f+u^\beta f^{-1}\langle\nabla F,\nabla u\rangle.
\end{align}
Combining $\Delta_p u+au^q=0$ and \eqref{3.25} together, we obtain
\begin{align}\label{3.28}
\left(\sum_{i=2}^nu_{ii}\right)^2=\left[(p-1)u_{11}+af^{1-\frac{p}{2}}u^q\right]^2.
\end{align}
By substituting \eqref{3.27} into \eqref{3.28}, we have
\begin{align}\label{3.29}
\left(\sum_{i=2}^nu_{ii}\right)^2=\left[\frac{p-1}{2}\beta u^{-1}f+\frac{p-1}{2}u^\beta f^{-1}\langle\nabla F,\nabla u\rangle+af^{1-\frac{p}{2}}u^q\right]^2.
\end{align}
By omitting some non-negative terms in $\left|\nabla\nabla u\right|^2$ and using Cauchy inequality, we arrive at
\begin{align}\label{3.30}
\left|\nabla\nabla u\right|^2\ge\sum_{i=1}^nu_{1i}^2+\sum_{i=2}^nu_{ii}^2\ge\sum_{i=1}^nu_{1i}^2+\frac{1}{n-1}\left(\sum_{i=2}^nu_{ii}\right)^2.
\end{align}
We plug \eqref{3.26} and \eqref{3.29} into \eqref{3.30} to obtain
\begin{align}\label{3.31}
\begin{split}
\left|\nabla\nabla u\right|^2\ge&\frac{1}{4}u^{2\beta}f^{-1}|\nabla F|^2+\frac{\beta^2}{4}u^{-2}f^2+\frac{\beta}{2} u^{\beta-1}\langle\nabla F,\nabla u\rangle\\
&+\frac{1}{n-1}\left[\frac{p-1}{2}\beta u^{-1}f+\frac{p-1}{2}u^\beta f^{-1}\langle\nabla F,\nabla u\rangle+af^{1-\frac{p}{2}}u^q\right]^2.
\end{split}
\end{align}
Hence, by expanding the last term of \eqref{3.31}, we obtain
\begin{align*}
\begin{split}
\left|\nabla\nabla u\right|^2\ge&\frac{1}{4}u^{2\beta}f^{-1}|\nabla F|^2+\frac{(p-1)^2}{4(n-1)}u^{2\beta}f^{-2}\langle\nabla F,\nabla u\rangle^2+\frac{\beta^2}{4}\left[1+\frac{(p-1)^2}{n-1}\right]u^{-2}f^2\\
&+\frac{\beta}{2}\left[1+\frac{(p-1)^2}{n-1}\right]u^{\beta-1}\langle\nabla F,\nabla u\rangle+\frac{a^2}{n-1}u^{2q}f^{2-p}+a\frac{p-1}{n-1}\beta u^{q-1}f^{2-\frac{p}{2}}\\
&+a\frac{p-1}{n-1}\beta u^{q-1}f^{2-\frac{p}{2}}+a\frac{p-1}{n-1}u^{q+\beta}f^{-\frac{p}{2}}\langle\nabla F,\nabla u\rangle.
\end{split}
\end{align*}
By using the above inequality and $\mathrm{Ric}\ge-(n-1)\kappa$, we have
\begin{align}\label{3.32}
\begin{split}
&2u^{-2\beta}f^{\frac{p}{2}}\left(\left|\nabla\nabla u\right|^2+\mathrm{Ric}\left(\nabla u,\nabla u\right)\right)\\
\ge&-2(n-1)\kappa u^{-2\beta}f^{\frac{p}{2}+1}+\frac{1}{2}f^{\frac{p}{2}-1}|\nabla F|^2+\beta\left[1+\frac{(p-1)^2}{n-1}\right]u^{-\beta-1}f^{\frac{p}{2}}\langle\nabla F,\nabla u\rangle\\
&+\frac{\beta^2}{2}\left[1+\frac{(p-1)^2}{n-1}\right]u^{-2\beta-2}f^{\frac{p}{2}+2}+\frac{(p-1)^2}{2(n-1)}f^{\frac{p}{2}-2}\langle\nabla F,\nabla u\rangle^2+\frac{2a^2}{n-1}u^{2q-2\beta}f^{2-\frac{p}{2}}\\
&+a\frac{2(p-1)}{n-1}\beta u^{q-2\beta-1}f^{2}+a\frac{2(p-1)}{n-1}u^{q-\beta}\langle\nabla F,\nabla u\rangle.
\end{split}
\end{align}
Substituting \eqref{3.32} into \eqref{3.17} yields
\begin{align}\label{3.33}
\begin{split}
\frac{1}{\alpha} F^{2-\alpha}\mathcal{L}_p(F^\alpha)\ge&-2(n-1)\kappa u^{-2\beta}f^{\frac{p}{2}+1}+2a\left[\frac{\beta}{2}\frac{n+1}{n-1}(p-1)-q\right]u^{q-2\beta-1}f^2\\
&+\left(\alpha+\frac{p}{2}-\frac{3}{2}\right)f^{\frac{p}{2}-1}|\nabla F|^2+\left[(\alpha-1)(p-2)+\frac{(p-1)^2}{2(n-1)}\right]f^{\frac{p}{2}-2}\langle\nabla F,\nabla u\rangle^2\\
&+(p-1)\beta\left[\frac{p-n}{2(n-1)}\beta+1\right]u^{-2\beta-2}f^{\frac{p}{2}+2}+(p-1)\beta\frac{p-n}{n-1}u^{-\beta-1}f^{\frac{p}{2}}\langle\nabla F,\nabla u\rangle\\
&+\frac{2a^2}{n-1}u^{2q-2\beta}f^{2-\frac{p}{2}}+a\frac{2(p-1)}{n-1}u^{q-\beta}\langle\nabla F,\nabla u\rangle.
\end{split}
\end{align}
By using $p>1$, $\alpha>1$ and
$$f^{\frac{p}{2}-1}|\nabla F|^2\ge f^{\frac{p}{2}-2}\langle\nabla F,\nabla u\rangle^2,$$
we arrive at
\begin{align}\label{3.34}
\begin{split}
&\left(\alpha+\frac{p}{2}-\frac{3}{2}\right)f^{\frac{p}{2}-1}|\nabla F|^2+\left[(\alpha-1)(p-2)+\frac{(p-1)^2}{2(n-1)}\right]f^{\frac{p}{2}-2}\langle\nabla F,\nabla u\rangle^2\\
\ge\,&(p-1)\left[\alpha+\frac{p-n}{2(n-1)}\right]f^{\frac{p}{2}-2}\langle\nabla F,\nabla u\rangle^2.
\end{split}
\end{align}
By substituting \eqref{3.34} into \eqref{3.33}, we obtain
\begin{align}\label{3.35}
\begin{split}
\frac{1}{\alpha} F^{2-\alpha}\mathcal{L}_p(F^\alpha)\ge&-2(n-1)\kappa u^{-2\beta}f^{\frac{p}{2}+1}+2a\left[\frac{\beta}{2}\frac{n+1}{n-1}(p-1)-q\right]u^{q-2\beta-1}f^2\\
&+(p-1)\left[\alpha+\frac{p-n}{2(n-1)}\right]f^{\frac{p}{2}-2}\langle\nabla F,\nabla u\rangle^2\\
&+(p-1)\beta\left[\frac{p-n}{2(n-1)}\beta+1\right]u^{-2\beta-2}f^{\frac{p}{2}+2}\\
&+(p-1)\beta\frac{p-n}{n-1}u^{-\beta-1}f^{\frac{p}{2}}\langle\nabla F,\nabla u\rangle+\frac{2a^2}{n-1}u^{2q-2\beta}f^{2-\frac{p}{2}}\\
&+a\frac{2(p-1)}{n-1}u^{q-\beta}\langle\nabla F,\nabla u\rangle.
\end{split}
\end{align}
Noting $\alpha>1$ and using the inequality $a_1^2-2a_1a_2\ge-a_2^2$, we infer
\begin{align}\label{3.36}
\begin{split}
&(p-1)\left[\alpha+\frac{p-n}{2(n-1)}\right]f^{\frac{p}{2}-2}\langle\nabla F,\nabla u\rangle^2+a\frac{2(p-1)}{n-1}u^{q-\beta}\langle\nabla F,\nabla u\rangle\\
\ge\,&-\frac{a^2}{(n-1)^2}(p-1)\left[\alpha+\frac{p-n}{2(n-1)}\right]^{-1}u^{2q-2\beta}f^{2-\frac{p}{2}}.
\end{split}
\end{align}
We substitute \eqref{3.36} into \eqref{3.35} to obtain
\begin{align}\label{3.37}
\begin{split}
\frac{1}{\alpha} F^{2-\alpha}\mathcal{L}(F^\alpha)\ge&-2(n-1)\kappa u^{-2\beta}f^{\frac{p}{2}+1}+2a\left[\frac{\beta}{2}\frac{n+1}{n-1}(p-1)-q\right]u^{q-2\beta-1}f^2\\
&+(p-1)\beta\left[\frac{p-n}{2(n-1)}\beta+1\right]u^{-2\beta-2}f^{\frac{p}{2}+2}+(p-1)\beta\frac{p-n}{n-1}u^{-\beta-1}f^{\frac{p}{2}}\langle\nabla F,\nabla u\rangle\\
&+\left\{\frac{2a^2}{n-1}-\frac{a^2}{(n-1)^2}(p-1)\left[\alpha+\frac{p-n}{2(n-1)}\right]^{-1}\right\}u^{2q-2\beta}f^{2-\frac{p}{2}}.
\end{split}
\end{align}
Noting $\alpha>1$ and using again the inequality $a_1^2-2a_1a_2\ge-a_2^2$, we obtain
\begin{align}\label{3.38}
\begin{split}
&\left\{\frac{2a^2}{n-1}-\frac{a^2}{(n-1)^2}(p-1)\left[\alpha+\frac{p-n}{2(n-1)}\right]^{-1}\right\}u^{2q-2\beta}f^{2-\frac{p}{2}}\\
&+2a\left[\frac{\beta}{2}\frac{n+1}{n-1}(p-1)-q\right]u^{q-2\beta-1}f^2\\
\ge &-\frac{\left[\frac{\beta}{2}\frac{n+1}{n-1}(p-1)-q\right]^2}{\frac{2}{n-1}-\frac{p-1}{(n-1)^2}
\left[\alpha+\frac{p-n}{2(n-1)}\right]^{-1}}u^{-2\beta-2}f^{\frac{p}{2}+2}.
\end{split}
\end{align}
Now, we plug \eqref{3.38} into \eqref{3.37} to deduce the desired inequality and hence complete the proof of Lemma \ref{lemma3.4}.
\end{proof}

\begin{lemma}\label{lemma3.5}
Let $\alpha>1$ and $v$ be a positive solution of equation \eqref{1} in $\Omega\subset M$ with $\rm{Ric}\ge-(n-1)\kappa$. Then, the following
\begin{align}\label{3.39}
\begin{split}
\frac{1}{\alpha} F^{2-\alpha}\mathcal{L}_p(F^\alpha)\ge&-2(n-1)\kappa u^{-2\beta}f^{\frac{p}{2}+1}+\beta(p-1)\frac{p-n}{n-1}u^{-\beta-1}f^\frac{p}{2}\langle\nabla F,\nabla u\rangle\\
&+2a\left[\frac{\beta}{2}\frac{n+1}{n-1}(p-1)-q\right]u^{q-2\beta-1}f^2+\frac{\beta}{2}(p-1)\left(2+\frac{p-n}{n-1}\beta\right)u^{-2\beta-2}f^{\frac{p}{2}+2}\\
\end{split}
\end{align}
holds point-wisely in $\{x\in\Omega:~f(x)>0\}$.
\end{lemma}
\begin{proof}
By using the inequality $(a_1+a_2)^2\ge a_1^2+2a_1a_2$, we have
\begin{align}\label{3.40}
\begin{split}
&\left[\frac{p-1}{2}\beta u^{-1}f+\frac{p-1}{2}u^\beta f^{-1}\langle\nabla F,\nabla u\rangle+af^{1-\frac{p}{2}}u^q\right]^2\\
\ge&\frac{(p-1)^2}{4}\beta^2u^{-2}f^2+(p-1)\beta u^{-1}f\left(\frac{p-1}{2}u^\beta f^{-1}\langle\nabla F,\nabla u\rangle+af^{1-\frac{p}{2}}u^q\right).
\end{split}
\end{align}
Substituting \eqref{3.40} into \eqref{3.31}, we have
\begin{align}\label{3.41}
\begin{split}
\left|\nabla\nabla u\right|^2\ge&\frac{1}{4}u^{2\beta}f^{-1}|\nabla F|^2+\frac{\beta^2}{4}\left[1+\frac{(p-1)^2}{n-1}\right]u^{-2}f^2+\frac{\beta}{2}\left[1+\frac{(p-1)^2}{n-1}\right]
u^{\beta-1}\langle\nabla F,\nabla u\rangle\\
&+a\beta\frac{p-1}{n-1}u^{q-1}f^{2-\frac{p}{2}}.
\end{split}
\end{align}
By using \eqref{3.41} and assumption $\mathrm{Ric}\ge-(n-1)\kappa$, we have
\begin{align}\label{3.42}
\begin{split}
&\,2u^{-2\beta}f^{\frac{p}{2}}\left(\left|\nabla\nabla u\right|^2+\mathrm{Ric}\left(\nabla u,\nabla u\right)\right)\\
\ge&-2(n-1)\kappa u^{-2\beta}f^{\frac{p}{2}+1}+\frac{1}{2}f^{\frac{p}{2}-1}|\nabla F|^2+\frac{\beta^2}{2}\left[1+\frac{(p-1)^2}{n-1}\right]u^{-2\beta-2}f^{\frac{p}{2}+2}\\
&+\beta\left[1+\frac{(p-1)^2}{n-1}\right] u^{-\beta-1}f^\frac{p}{2}\langle\nabla F,\nabla u\rangle+2a\beta\frac{p-1}{n-1}u^{q-2\beta-1}f^{2}.
\end{split}
\end{align}
We plug \eqref{3.42} into \eqref{3.17} to derive
\begin{align}\label{3.43}
\begin{split}
\frac{1}{\alpha} F^{2-\alpha}\mathcal{L}_p(F^\alpha)\ge&-2(n-1)\kappa u^{-2\beta}f^{\frac{p}{2}+1}+\beta(p-1)\frac{p-n}{n-1}u^{-\beta-1}f^\frac{p}{2}\langle\nabla F,\nabla u\rangle\\
&+\left(\alpha+\frac{p}{2}-\frac{3}{2}\right)f^{\frac{p}{2}-1}|\nabla F|^2+(\alpha-1)(p-2)f^{\frac{p}{2}-2}\langle\nabla F,\nabla u\rangle^2\\
&+2a\left[\frac{\beta}{2}\frac{n+1}{n-1}(p-1)-q\right]u^{q-2\beta-1}f^2
+\frac{\beta}{2}(p-1)\left(2+\frac{p-n}{n-1}\beta\right)u^{-2\beta-2}f^{\frac{p}{2}+2}.\\
\end{split}
\end{align}
Noting $p>1$, $\alpha>1$ and
$$f^{\frac{p}{2}-1}|\nabla F|^2\ge f^{\frac{p}{2}-2}\langle\nabla F,\nabla u\rangle^2,$$
we arrive at
\begin{align}\label{3.44}
\left(\alpha+\frac{p}{2}-\frac{3}{2}\right)f^{\frac{p}{2}-1}|\nabla F|^2+(\alpha-1)(p-2)f^{\frac{p}{2}-2}\langle\nabla F,\nabla u\rangle^2\ge\left(\alpha-\frac{1}{2}\right)(p-1)f^{\frac{p}{2}-2}\langle\nabla F,\nabla u\rangle^2.
\end{align}
In view of \eqref{3.43} and \eqref{3.44}, we can derive the desired inequality and complete the proof of Lemma \ref{lemma3.5}.
\end{proof}

By using Lemma \ref{lemma3.4} and Lemma \ref{lemma3.5}, we can achieve the following point-wise estimate of $\mathcal{L}_p(F^\alpha)$.
\begin{lemma}\label{lemma3.6}
Let $v$ be a positive solution of equation \eqref{1} in $\Omega\subset M$ with $\rm{Ric}\ge-(n-1)\kappa$. Set
\begin{align}\label{3.45}
h=\beta(p-1)\left[\frac{p-n}{(n-1)^2}\beta+\frac{2}{n-1}\right]
\end{align}
and suppose that $\beta$ satisfies the following condition
\begin{align}\label{3.46}
\begin{split}
\begin{cases}
\beta\in\left(0,~\frac{2(n-1)}{n-p}\right),&1<p<n,\\[3mm]
\beta\in(0,+\infty),&p\ge n.
\end{cases}
\end{split}
\end{align}
If the constants $a$, $p$ and $q$ satisfy one of the following two conditions
\begin{align}
\frac{\beta}{2}\cdot\frac{n+1}{n-1}(p-1)-h^{\frac{1}{2}}<q<\frac{\beta}{2}\cdot\frac{n+1}{n-1}(p-1)+h^{\frac{1}{2}},\quad (a\neq 0);
\end{align}
\begin{align}
a\left[\frac{\beta}{2}\cdot\frac{n+1}{n-1}(p-1)-q\right]\ge0,
\end{align}
then there exists $\alpha>1$ such that the following
\begin{align}\label{3.49}
\frac{1}{\alpha}\mathcal{L}(F^\alpha)\ge-2(n-1)\kappa u^{\left(\frac{p}{2}-1\right)\beta}F^{\alpha+\frac{p}{2}-1}+ Au^{\frac{p}{2}\beta-2}F^{\alpha+\frac{p}{2}}- Bu^{\frac{\beta}{2}(p-1)-1}F^{\alpha+\frac{p}{2}-\frac{3}{2}}|\nabla F|
\end{align}
holds point-wisely in $\{x\in\Omega:~f(x)>0\}$. Furthermore, $A$ is a positive constant and
$$B=(p-1)\beta\frac{|n-p|}{n-1}.$$
\end{lemma}

\begin{proof}
\textbf{Case $($\romannumeral 1$):$}
\begin{align*}
\frac{\beta}{2}\cdot\frac{n+1}{n-1}(p-1)-h^{\frac{1}{2}}<q<\frac{\beta}{2}\cdot\frac{n+1}{n-1}(p-1)+h^{\frac{1}{2}},\quad (a\neq 0).
\end{align*}
This implies
\begin{align*}
\lim_{\alpha\to+\infty}H>0,
\end{align*}
where $H$ is defined in \eqref{3.21}. Thus, we can choose $\alpha_0$ large enough such that for any $\alpha>\alpha_0$,
\begin{align}\label{3.50}
H>0.
\end{align}
Furthermore, we have
\begin{align}\label{3.51}
\begin{split}
    (p-1)\beta\frac{p-n}{n-1}u^{-\beta-1}f^{\frac{p}{2}}\langle\nabla F,\nabla u\rangle &\ge-(p-1)\beta\frac{|n-p|}{n-1}u^{-\beta-1}f^{\frac{p}{2}+\frac{1}{2}}|\nabla F|\\
&=-Bu^{-\beta-1}f^{\frac{p}{2}+\frac{1}{2}}|\nabla F|.
\end{split}
\end{align}
Combining \eqref{3.22} and \eqref{3.51} together, we can infer that
\begin{align}\label{3.52}
\frac{1}{\alpha} F^{2-\alpha}\mathcal{L}(F^\alpha)\ge-2(n-1)\kappa u^{-2\beta}f^{\frac{p}{2}+1}-Bu^{-\beta-1}f^{\frac{p}{2}+\frac{1}{2}}|\nabla F|+Hu^{-2\beta-2}f^{\frac{p}{2}+2},\quad(a\neq0).
\end{align}
Combining \eqref{3.50} and \eqref{3.52} together leads to \eqref{3.49}.
\medskip

\textbf{Case $($\romannumeral 2$):$}
\begin{align*}
a\left[\frac{\beta}{2}\frac{n+1}{n-1}(p-1)-q\right]\ge0.
\end{align*}
Under the above condition, we have
\begin{align}\label{3.53}
2a\left[\frac{\beta}{2}\frac{n+1}{n-1}(p-1)-q\right]u^{q-2\beta-1}f^2\ge0.
\end{align}
Combining \eqref{3.39}, \eqref{3.51} and \eqref{3.53}, we obtain
\begin{align}\label{3.54}
\begin{split}
\frac{1}{\alpha} F^{2-\alpha}\mathcal{L}_p(F^\alpha)\ge&-2(n-1)\kappa u^{-2\beta}f^{\frac{p}{2}+1}-Bu^{-\beta-1}f^{\frac{p}{2}+\frac{1}{2}}|\nabla F|\\
&+\frac{\beta}{2}(p-1)\left(2+\frac{p-n}{n-1}\beta\right)u^{-2\beta-2}f^{\frac{p}{2}+2}.\\
\end{split}
\end{align}
Since $\beta$ satisfies \eqref{3.46}, it is easy to see that
\begin{align*}
\frac{\beta}{2}(p-1)\left(2+\frac{p-n}{n-1}\beta\right)>0.
\end{align*}
Combining above, we complete the proof of Lemma \ref{lemma3.6}.
\end{proof}

\subsection{Deducing the main integral inequality}\

\begin{lemma}
Let $M$ be a complete manifold with $\rm{Ric}\ge-(n-1)\kappa$ and $v$ be a positive solution of equation \eqref{1} in $\Omega=B_R(x_0)\subset M$. Then, there exist constants $t$ large enough and $\mu_1>0$ such that
\begin{align*}
\begin{split}
&\exp\left\{-C_n(1+\sqrt{\kappa}R)\right\}V^\frac{2}{n}R^{-2}\left\|F^{\frac{p}{4}+\frac{\alpha-1}{2}
+\frac{t}{2}}\eta\right\|^2_{L^\frac{2n}{n-2}(\Omega)}+\mu t\phi_\beta^{\beta-2}\int_\Omega F^{t+\alpha+\frac{p}{2}}\eta^2\\
\le&\left[(n-1)\mu_1t\kappa+ \frac{1}{R^2}\right] \int_\Omega F^{t+\alpha+\frac{p}{2}-1}\eta^2+\mu_1\int_\Omega F^{t+\frac{p}{2}+\alpha-1}|\nabla\eta|^2,
\end{split}
\end{align*}
where $\eta\in C_0^\infty(\Omega,\mathbb{R})$ is non-negative function and $V$ is the volume of $B_R(x_0)$.
\end{lemma}

\begin{proof}
Now we choose a geodesic ball $\Omega=B_R(x_0)\subset M$. If we choose a test function $\xi\cdot u^\lambda=F^t_\varepsilon\eta^2\cdot u^\lambda$ where $\eta\in C_0^\infty(\Omega,\mathbb{R})$ is non-negative, $F_\varepsilon=(F-\varepsilon)^+$, $\varepsilon>0$, $t>1$ and $\lambda\in\mathbb{R}$ are to be determined later. It follows from \eqref{2.1} that
\begin{align*}
    &\frac{1}{\alpha}\int_\Omega\mathcal{L}_p(F^\alpha)\cdot \xi \cdot u^\lambda\\
    =&-\int_\Omega\left\langle\nabla(\xi u^\lambda),f^{\frac{p}{2}-1}F^{\alpha-1}\left[\nabla F+(p-2)f^{-1}\langle\nabla u,\nabla F\rangle\nabla u\right]\right\rangle\\
    =&-\int_\Omega f^{\frac{p}{2}-1}F^{\alpha-1}u^\lambda\langle\nabla F,\nabla \xi\rangle-\lambda\int_\Omega u^{\lambda-1}f^{\frac{p}{2}-1}F^{\alpha-1}\langle\nabla F,\nabla u\rangle\xi\\
    &-(p-2)\int_\Omega f^{\frac{p}{2}-2}F^{\alpha-1}u^\lambda\langle\nabla F,\nabla u\rangle\langle\nabla u,\nabla\xi\rangle-(p-2)\lambda\int_\Omega f^{\frac{p}{2}-1}F^{\alpha-1}u^{\lambda-1}\langle\nabla F,\nabla u\rangle\xi\\
    =&-\int_\Omega f^{\frac{p}{2}-1}F^{\alpha-1}u^\lambda\langle\nabla F,\nabla \xi\rangle-(p-1)\lambda\int_\Omega f^{\frac{p}{2}-1}F^{\alpha-1}u^{\lambda-1}\langle\nabla F,\nabla u\rangle\xi\\
    &-(p-2)\int_\Omega f^{\frac{p}{2}-2}F^{\alpha-1}u^\lambda\langle\nabla F,\nabla u\rangle\langle\nabla u,\nabla\xi\rangle.
\end{align*}
Since $\xi=F^t_\varepsilon\eta^2$, we can achieve that
\begin{align}\label{3.55}
    \begin{split}
        &\frac{1}{\alpha}\int_\Omega\mathcal{L}_p(F^\alpha)\cdot F^t_\varepsilon\eta^2 \cdot u^\lambda\\
        =&-\int_\Omega f^{\frac{p}{2}-1}F^{\alpha-1}u^\lambda\langle\nabla F,tF^{t-1}_\varepsilon\eta^2\nabla F+2F^t_\varepsilon\eta\nabla\eta\rangle-(p-1)\lambda\int_\Omega f^{\frac{p}{2}-1}F^{\alpha-1}u^{\lambda-1}\langle\nabla F,\nabla u\rangle F^t_\varepsilon\eta^2\\
    &-(p-2)\int_\Omega f^{\frac{p}{2}-2}F^{\alpha-1}u^\lambda\langle\nabla F,\nabla u\rangle\langle\nabla u,tF^{t-1}_\varepsilon\eta^2\nabla F+2F^t_\varepsilon\eta\nabla\eta\rangle\\
     =&-t\int_\Omega u^{\left(\frac{p}{2}-1\right)\beta+\lambda}F^{\frac{p}{2}+\alpha-2}F^{t-1}_\varepsilon |\nabla F|^2\eta^2-2\int_\Omega u^{\left(\frac{p}{2}-1\right)\beta+\lambda}F^{\frac{p}{2}+\alpha-2}F^{t}_\varepsilon\langle\nabla F,\nabla\eta\rangle\eta\\
    &-(p-1)\lambda\int_\Omega u^{\left(\frac{p}{2}-1\right)\beta+\lambda-1}F^{\frac{p}{2}+\alpha-2}F^{t}_\varepsilon\langle\nabla F,\nabla u\rangle\eta^2\\
    &-(p-2)t\int_\Omega u^{\left(\frac{p}{2}-2\right)\beta+\lambda}F^{\frac{p}{2}+\alpha-3}F^{t-1}_\varepsilon\langle\nabla F,\nabla u\rangle^2\eta^2\\
    &-2(p-2)\int_\Omega u^{\left(\frac{p}{2}-2\right)\beta+\lambda}F^{\frac{p}{2}+\alpha-3}F^{t}_\varepsilon\langle\nabla F,\nabla u\rangle\langle\nabla u,\nabla\eta\rangle\eta.
    \end{split}
\end{align}
Combining \eqref{3.49} and \eqref{3.55}, we can achieve the following inequality.
\begin{align*}
        &-t\int_\Omega u^{\left(\frac{p}{2}-1\right)\beta+\lambda}F^{\frac{p}{2}+\alpha-2}F^{t-1}_\varepsilon |\nabla F|^2\eta^2-2\int_\Omega u^{\left(\frac{p}{2}-1\right)\beta+\lambda}F^{\frac{p}{2}+\alpha-2}F^{t}_\varepsilon\langle\nabla F,\nabla\eta\rangle\eta\\
    &-(p-1)\lambda\int_\Omega u^{\left(\frac{p}{2}-1\right)\beta+\lambda-1}F^{\frac{p}{2}+\alpha-2}F^{t}_\varepsilon\langle\nabla F,\nabla u\rangle\eta^2\\
    &-(p-2)t\int_\Omega u^{\left(\frac{p}{2}-2\right)\beta+\lambda}F^{\frac{p}{2}+\alpha-3}F^{t-1}_\varepsilon\langle\nabla F,\nabla u\rangle^2\eta^2\\
    &-2(p-2)\int_\Omega u^{\left(\frac{p}{2}-2\right)\beta+\lambda}F^{\frac{p}{2}+\alpha-3}F^{t}_\varepsilon\langle\nabla F,\nabla u\rangle\langle\nabla u,\nabla\eta\rangle\eta\\
    \ge&-2(n-1)\kappa \int_\Omega u^{\left(\frac{p}{2}-1\right)\beta+\lambda}F^{\alpha+\frac{p}{2}-1}F^t_\varepsilon\eta^2+ A\int_\Omega u^{\frac{p}{2}\beta-2+\lambda}F^{\alpha+\frac{p}{2}}F^t_\varepsilon\eta^2\\
    &- B\int_\Omega u^{\frac{\beta}{2}(p-1)-1+\lambda}F^{\alpha+\frac{p}{2}-\frac{3}{2}}|\nabla F|F^t_\varepsilon\eta^2.
\end{align*}
By using the above inequality, we can obtain the following
\begin{align}\label{3.57}
\begin{split}
&2(n-1)\kappa \int_\Omega u^{\left(\frac{p}{2}-1\right)\beta+\lambda}F^{\alpha+\frac{p}{2}-1}F^t_\varepsilon\eta^2-2\int_\Omega u^{\left(\frac{p}{2}-1\right)\beta+\lambda}F^{\frac{p}{2}+\alpha-2}F^{t}_\varepsilon\langle\nabla F,\nabla\eta\rangle\eta\\
&-2(p-2)\int_\Omega u^{\left(\frac{p}{2}-2\right)\beta+\lambda}F^{\frac{p}{2}+\alpha-3}F^{t}_\varepsilon\langle\nabla F,\nabla u\rangle\langle\nabla u,\nabla\eta\rangle\eta\\
&+\left[B+(p-1)|\lambda|\right]\int_\Omega u^{\frac{\beta}{2}(p-1)-1+\lambda}F^{\alpha+\frac{p}{2}-\frac{3}{2}}|\nabla F|F^t_\varepsilon\eta^2\\
\ge\,&t\int_\Omega u^{\left(\frac{p}{2}-1\right)\beta+\lambda}F^{\frac{p}{2}+\alpha-2}F^{t-1}_\varepsilon |\nabla F|^2\eta^2+(p-2)t\int_\Omega u^{\left(\frac{p}{2}-2\right)\beta+\lambda}F^{\frac{p}{2}+\alpha-3}F^{t-1}_\varepsilon\langle\nabla F,\nabla u\rangle^2\eta^2\\
&+ A\int_\Omega u^{\frac{p}{2}\beta-2+\lambda}F^{\alpha+\frac{p}{2}}F^t_\varepsilon\eta^2.\\
\end{split}
\end{align}
Set
\begin{align}\label{3.58}
L_p=t\int_\Omega u^{\left(\frac{p}{2}-1\right)\beta+\lambda}F^{\frac{p}{2}+\alpha-2}F^{t-1}_\varepsilon |\nabla F|^2\eta^2+(p-2)t\int_\Omega u^{\left(\frac{p}{2}-2\right)\beta+\lambda}F^{\frac{p}{2}+\alpha-3}F^{t-1}_\varepsilon\langle\nabla F,\nabla u\rangle^2\eta^2.
\end{align}
Then we need to consider the following two cases:

($\romannumeral 1$) $p\ge2$;\quad\quad($\romannumeral 2$) $1<p<2$.
\medskip

\noindent For the case ($\romannumeral 1$), from \eqref{3.35} we can easily know that
\begin{align}\label{3.59}
L_p\ge t\int_\Omega u^{\left(\frac{p}{2}-1\right)\beta+\lambda}F^{\frac{p}{2}+\alpha-2}F^{t-1}_\varepsilon |\nabla F|^2\eta^2.
\end{align}

\noindent For the case ($\romannumeral 2$),  from \eqref{3.35} we can achieve that
\begin{align}\label{3.60}
\begin{split}
L_p=&t\int_\Omega u^{\left(\frac{p}{2}-1\right)\beta+\lambda}F^{\frac{p}{2}+\alpha-2}F^{t-1}_\varepsilon |\nabla F|^2\eta^2-(2-p)t\int_\Omega u^{\left(\frac{p}{2}-2\right)\beta+\lambda}F^{\frac{p}{2}+\alpha-3}F^{t-1}_\varepsilon\langle\nabla F,\nabla u\rangle^2\eta^2\\
\ge&t\int_\Omega u^{\left(\frac{p}{2}-1\right)\beta+\lambda}F^{\frac{p}{2}+\alpha-2}F^{t-1}_\varepsilon |\nabla F|^2\eta^2-(2-p)t\int_\Omega u^{\left(\frac{p}{2}-1\right)\beta+\lambda}F^{\frac{p}{2}+\alpha-2}F^{t-1}_\varepsilon |\nabla F|^2\eta^2\\
=&(p-1)t\int_\Omega u^{\left(\frac{p}{2}-1\right)\beta+\lambda}F^{\frac{p}{2}+\alpha-2}F^{t-1}_\varepsilon |\nabla F|^2\eta^2.
\end{split}
\end{align}
Set
\begin{align}\label{3.61}
\begin{split}
\theta(p)=
\begin{cases}
p-1,&1<p<2;\\[3mm]
1,&p\ge2.
\end{cases}
\end{split}
\end{align}
Therefore, combining \eqref{3.58}, \eqref{3.59}, \eqref{3.60} and \eqref{3.61} yields
\begin{align}\label{3.62}
L_p\ge \theta(p)t\int_\Omega u^{\left(\frac{p}{2}-1\right)\beta+\lambda}F^{\frac{p}{2}+\alpha-2}F^{t-1}_\varepsilon |\nabla F|^2\eta^2.
\end{align}
By combining \eqref{3.57} and \eqref{3.62} together, we have
\begin{align*}
&2(n-1)\kappa \int_\Omega u^{\left(\frac{p}{2}-1\right)\beta+\lambda}F^{\alpha+\frac{p}{2}-1}F^t_\varepsilon\eta^2-2\int_\Omega u^{\left(\frac{p}{2}-1\right)\beta+\lambda}F^{\frac{p}{2}+\alpha-2}F^{t}_\varepsilon\langle\nabla F,\nabla\eta\rangle\eta\\
&-2(p-2)\int_\Omega u^{\left(\frac{p}{2}-2\right)\beta+\lambda}F^{\frac{p}{2}+\alpha-3}F^{t}_\varepsilon\langle\nabla F,\nabla u\rangle\langle\nabla u,\nabla\eta\rangle\eta\\
&+\left[B+(p-1)|\lambda|\right]\int_\Omega u^{\frac{\beta}{2}(p-1)-1+\lambda}F^{\alpha+\frac{p}{2}-\frac{3}{2}}|\nabla F|F^t_\varepsilon\eta^2\\
\ge&\theta(p)t\int_\Omega u^{\left(\frac{p}{2}-1\right)\beta+\lambda}F^{\frac{p}{2}+\alpha-2}F^{t-1}_\varepsilon |\nabla F|^2\eta^2+ A\int_\Omega u^{\frac{p}{2}\beta-2+\lambda}F^{\alpha+\frac{p}{2}}F^t_\varepsilon\eta^2.
\end{align*}
By letting $\varepsilon\longrightarrow 0^+$, we obtain
\begin{align}\label{3.63}
\begin{split}
&2(n-1)\kappa \int_\Omega u^{\left(\frac{p}{2}-1\right)\beta+\lambda}F^{t+\alpha+\frac{p}{2}-1}\eta^2-2\int_\Omega u^{\left(\frac{p}{2}-1\right)\beta+\lambda}F^{t+\frac{p}{2}+\alpha-2}\langle\nabla F,\nabla\eta\rangle\eta\\
&-2(p-2)\int_\Omega u^{\left(\frac{p}{2}-2\right)\beta+\lambda}F^{t+\frac{p}{2}+\alpha-3}\langle\nabla F,\nabla u\rangle\langle\nabla u,\nabla\eta\rangle\eta\\
&+\left[B+(p-1)|\lambda|\right]\int_\Omega u^{\frac{\beta}{2}(p-1)-1+\lambda}F^{t+\alpha+\frac{p}{2}-\frac{3}{2}}|\nabla F|\eta^2\\
\ge&\theta(p)t\int_\Omega u^{\left(\frac{p}{2}-1\right)\beta+\lambda}F^{t+\frac{p}{2}+\alpha-3} |\nabla F|^2\eta^2+ A\int_\Omega u^{\frac{p}{2}\beta-2+\lambda}F^{t+\alpha+\frac{p}{2}}\eta^2.\\
\end{split}
\end{align}
By the absolute value inequality and Cauchy inequality, we have
\begin{align}\label{3.64}
\begin{split}
&\left[B+(p-1)|\lambda|\right]\int_\Omega u^{\frac{\beta}{2}(p-1)-1+\lambda}F^{t+\alpha+\frac{p}{2}-\frac{3}{2}}|\nabla F|\eta^2\\
\le&\frac{\theta(p)}{4}t\int_\Omega u^{\left(\frac{p}{2}-1\right)\beta+\lambda}F^{t+\frac{p}{2}+\alpha-3} |\nabla F|^2\eta^2 +\frac{1}{\theta(p)t}\left[B+(p-1)|\lambda|\right]^2\int_\Omega u^{\frac{p}{2}\beta-2+\lambda}F^{t+\alpha+\frac{p}{2}}\eta^2,
\end{split}
\end{align}

\begin{align}\label{3.65}
\begin{split}
&-2\int_\Omega u^{\left(\frac{p}{2}-1\right)\beta+\lambda}F^{t+\frac{p}{2}+\alpha-2}\langle\nabla F,\nabla\eta\rangle\eta\le2\int_\Omega u^{\left(\frac{p}{2}-1\right)\beta+\lambda}F^{t+\frac{p}{2}+\alpha-2}|\nabla F||\nabla\eta|\eta\\
\le&\frac{\theta(p)}{4}t\int_\Omega u^{\left(\frac{p}{2}-1\right)\beta+\lambda}F^{t+\frac{p}{2}+\alpha-3} |\nabla F|^2\eta^2+\frac{4}{\theta(p)t}\int_\Omega u^{\left(\frac{p}{2}-1\right)\beta+\lambda}F^{t+\frac{p}{2}+\alpha-1}|\nabla\eta|^2
\end{split}
\end{align}
and
\begin{align}\label{3.66}
\begin{split}
&-2(p-2)\int_\Omega u^{\left(\frac{p}{2}-2\right)\beta+\lambda}F^{t+\frac{p}{2}+\alpha-3}\langle\nabla F,\nabla u\rangle\langle\nabla u,\nabla\eta\rangle\eta\\
\le&2|p-2|\int_\Omega u^{\left(\frac{p}{2}-1\right)\beta+\lambda}F^{t+\frac{p}{2}+\alpha-2}|\nabla F||\nabla\eta|\eta\\
\le&\frac{\theta(p)}{4}t\int_\Omega u^{\left(\frac{p}{2}-1\right)\beta+\lambda}F^{t+\frac{p}{2}+\alpha-3} |\nabla F|^2\eta^2+\frac{4}{\theta(p)t}(p-2)^2\int_\Omega u^{\left(\frac{p}{2}-1\right)\beta+\lambda}F^{t+\frac{p}{2}+\alpha-1}|\nabla\eta|^2.
\end{split}
\end{align}
Substituting \eqref{3.64}, \eqref{3.65} and \eqref{3.66} into \eqref{3.63}, we obtain
\begin{align}\label{3.67}
\begin{split}
&\frac{\theta(p)}{4}t\int_\Omega u^{\left(\frac{p}{2}-1\right)\beta+\lambda}F^{t+\frac{p}{2}+\alpha-3} |\nabla F|^2\eta^2+A\int_\Omega u^{\frac{p}{2}\beta-2+\lambda}F^{t+\alpha+\frac{p}{2}}\eta^2\\
\le&2(n-1)\kappa \int_\Omega u^{\left(\frac{p}{2}-1\right)\beta+\lambda}F^{t+\alpha+\frac{p}{2}-1}\eta^2 +\frac{1}{\theta(p)t}\left[B+(p-1)|\lambda|\right]^2\int_\Omega u^{\frac{p}{2}\beta-2+\lambda}F^{t+\alpha+\frac{p}{2}}\eta^2\\
&+\frac{4}{\theta(p)t}\left[1+(p-2)^2\right]\int_\Omega u^{\left(\frac{p}{2}-1\right)\beta+\lambda}F^{t+\frac{p}{2}+\alpha-1}|\nabla\eta|^2.
\end{split}
\end{align}
Now we choose $t$ large enough such that
\begin{align}\label{3.68}
\frac{1}{\theta(p)t}\left[B+(p-1)|\lambda|\right]^2\le\frac{A}{2}.
\end{align}
It follows from \eqref{3.67} and \eqref{3.68} that
\begin{align*}
\begin{split}
&\frac{\theta(p)}{4}t\int_\Omega u^{\left(\frac{p}{2}-1\right)\beta+\lambda}F^{t+\frac{p}{2}+\alpha-3}|\nabla F|^2\eta^2+\frac{A}{2}\int_\Omega u^{\frac{p}{2}\beta-2+\lambda}F^{t+\alpha+\frac{p}{2}}\eta^2\\
\le&2(n-1)\kappa \int_\Omega u^{\left(\frac{p}{2}-1\right)\beta +\lambda}F^{t+\alpha+\frac{p}{2}-1}\eta^2 +\frac{4}{\theta(p)t}\left[1+(p-2)^2\right]\int_\Omega u^{\left(\frac{p}{2}-1\right)\beta+\lambda}F^{t+\frac{p}{2}+\alpha-1}|\nabla\eta|^2.
\end{split}
\end{align*}
By letting $\lambda=\beta\left(1-\frac{p}{2}\right)$, we obtain
\begin{align}\label{3.69}
\begin{split}
&\frac{\theta(p)}{4}t\int_\Omega F^{t+\frac{p}{2}+\alpha-3} |\nabla F|^2\eta^2+\frac{A}{2}\int_\Omega u^{\beta-2}F^{t+\alpha+\frac{p}{2}}\eta^2\\
\le&2(n-1)\kappa \int_\Omega F^{t+\alpha+\frac{p}{2}-1}\eta^2+\frac{4}{\theta(p)t}\left[1+(p-2)^2\right]\int_\Omega F^{t+\frac{p}{2}+\alpha-1}|\nabla\eta|^2.
\end{split}
\end{align}

On the other hand, we have
\begin{align}\label{3.70}
\begin{split}
\left|\nabla\right(F^{\frac{p}{4}+\frac{\alpha-1}{2}+\frac{t}{2}}\eta\left)\right|^2&=\left|\left(\frac{p+2t+2\alpha-2}{4}\right)F^{\frac{p}{4}
+\frac{t}{2}+\frac{\alpha-3}{2}}\eta\nabla F+F^{\frac{p}{4}+\frac{t}{2}+\frac{\alpha-1}{2}}\nabla\eta\right|^2\\
&\le\frac{(p+2t+2\alpha-2)^2}{8}F^{\frac{p}{2}+t+\alpha-3}\left|\nabla F\right|^2\eta^2+2F^{\frac{p}{2}+t+\alpha-1}\left|\nabla\eta\right|^2.
\end{split}
\end{align}
Substituting \eqref{3.70} into \eqref{3.69} gives
\begin{align}\label{3.71}
\begin{split}
&\frac{2\theta(p)t}{(p+2t+2\alpha-2)^2}\int_\Omega\left|\nabla\right(F^{\frac{p}{4}+\frac{\alpha-1}{2}+\frac{t}{2}}\eta\left)\right|^2+\frac{A}{2}\int_\Omega u^{\beta-2}F^{t+\alpha+\frac{p}{2}}\eta^2\\
\le&2(n-1)\kappa \int_\Omega F^{t+\alpha+\frac{p}{2}-1}\eta^2+\left\{\frac{4}{\theta(p)t}\left[1+(p-2)^2\right]+ \frac{4\theta(p)t}{(p+2t+2\alpha-2)^2}\right\}\int_\Omega F^{t+\frac{p}{2}+\alpha-1}|\nabla\eta|^2.
\end{split}
\end{align}
On the other hand, we note that Saloff-Coste's Sobolev inequality implies
\begin{align}\label{3.72}
\begin{split}
&\exp\left\{-C_n(1+\sqrt{\kappa}R)\right\}V^\frac{2}{n}R^{-2}\left\|F^{\frac{p}{4}+\frac{\alpha-1}{2}
+\frac{t}{2}}\eta\right\|^2_{L^\frac{2n}{n-2}(\Omega)}\\
\le&\int_\Omega{\left|\nabla\left(F^{\frac{p}{4}+\frac{\alpha-1}{2}+\frac{t}{2}}\eta\right)\right|^2}
+R^{-2}\int_\Omega{F^{\frac{p}{2}+\alpha+t-1}\eta^2}.
\end{split}
\end{align}
Now, we substitute \eqref{3.72} into \eqref{3.71} to obtain
\begin{align}\label{3.73}
\begin{split}
&\frac{2\theta(p)t}{(p+2t+2\alpha-2)^2}\exp\left\{-C_n(1+\sqrt{\kappa}R)\right\}V^\frac{2}{n}R^{-2}\left\|F^{\frac{p}{4}
+\frac{\alpha-1}{2}+\frac{t}{2}}\eta\right\|^2_{L^\frac{2n}{n-2}(\Omega)}\\
\le&-\frac{A}{2}\int_\Omega u^{\beta-2}F^{t+\alpha+\frac{p}{2}}\eta^2 +\left[2(n-1)\kappa+\frac{2\theta(p)t}{(p+2t+2\alpha-2)^2R^2}\right] \int_\Omega F^{t+\alpha+\frac{p}{2}-1}\eta^2\\
&+\left\{\frac{4}{\theta(p)t}\left[1+(p-2)^2\right] + \frac{4\theta(p)t}{(p+2t+2\alpha-2)^2}\right\}\int_\Omega F^{t+\frac{p}{2}+\alpha-1}|\nabla\eta|^2.
\end{split}
\end{align}
We divide the both sides of $\eqref{3.73}$ by
$$\frac{2\theta(p)t}{(p+2t+2\alpha-2)^2}$$
to obtain
\begin{align}\label{3.74}
\begin{split}
&\exp\left\{-C_n(1+\sqrt{\kappa}R)\right\}V^\frac{2}{n}R^{-2}\left\|F^{\frac{p}{4}+\frac{\alpha-1}{2}
+\frac{t}{2}}\eta\right\|^2_{L^\frac{2n}{n-2}(\Omega)}\\
&+\frac{A(p+2t+2\alpha-2)^2}{4\theta(p)t}\int_\Omega u^{\beta-2}F^{t+\alpha+\frac{p}{2}}\eta^2\\
\le &\left[(n-1)\frac{(p+2t+2\alpha-2)^2}{\theta(p)t}\kappa+ \frac{1}{R^2}\right] \int_\Omega F^{t+\alpha+\frac{p}{2}-1}\eta^2\\
&+\left\{\frac{2(p+2t+2\alpha-2)^2}{\theta^2(p)t^2}\left[1+(p-2)^2\right]+2\right\}\int_\Omega F^{t+\frac{p}{2}+\alpha-1}|\nabla\eta|^2.
\end{split}
\end{align}

Set
\begin{align}
\label{3.75}\mu_1&=\mathop{\sup}\limits_{t\in[1,\infty)}\frac{2(p+2t+2\alpha-2)^2}{\theta^2(p)t^2}\left[1+(p-2)^2\right]+2,\\
\label{3.76}\mu&=A\mathop{\inf}\limits_{t\in[1,\infty)}\frac{(p+2t+2\alpha-2)^2}{4\theta(p)t}.
\end{align}
Then, we can see easily that $\mu_1$ and $\mu$ are both finite positive constants. Combining \eqref{3.74}, \eqref{3.75} and \eqref{3.76} yields
\begin{align}\label{3.77}
\begin{split}
&\exp\left\{-C_n(1+\sqrt{\kappa}R)\right\}V^\frac{2}{n}R^{-2}\left\|F^{\frac{p}{4}+\frac{\alpha-1}{2}
+\frac{t}{2}}\eta\right\|^2_{L^\frac{2n}{n-2}(\Omega)}+\mu t\int_\Omega u^{\beta-2}F^{t+\alpha+\frac{p}{2}}\eta^2\\
\le&\left[(n-1)\mu_1t\kappa+ \frac{1}{R^2}\right] \int_\Omega F^{t+\alpha+\frac{p}{2}-1}\eta^2+\mu_1\int_\Omega F^{t+\frac{p}{2}+\alpha-1}|\nabla\eta|^2.
\end{split}
\end{align}
Set
\begin{align}\label{3.78}
\begin{split}
\phi_\beta=
\begin{cases}
\mathop{\sup}\limits_{\Omega}u,&0<\beta<2, \\[3mm]
1,&\beta=2, \\[3mm]
\mathop{\inf}\limits_{\Omega}u,&\beta>2.
\end{cases}
\end{split}
\end{align}
Hence, combining \eqref{3.77} and \eqref{3.78} together, we can obtain that
\begin{align}\label{3.79}
\begin{split}
&\exp\left\{-C_n(1+\sqrt{\kappa}R)\right\}V^\frac{2}{n}R^{-2}\left\|F^{\frac{p}{4}+\frac{\alpha-1}{2}
+\frac{t}{2}}\eta\right\|^2_{L^\frac{2n}{n-2}(\Omega)}+\mu t\phi_\beta^{\beta-2}\int_\Omega F^{t+\alpha+\frac{p}{2}}\eta^2\\
\le&\left[(n-1)\mu_1t\kappa+ \frac{1}{R^2}\right] \int_\Omega F^{t+\alpha+\frac{p}{2}-1}\eta^2+\mu_1\int_\Omega F^{t+\frac{p}{2}+\alpha-1}|\nabla\eta|^2.
\end{split}
\end{align}
Thus, we finish the proof of this lemma.
\end{proof}

\subsection{$L^{\beta_1}$-bound of gradient in a geodesic ball with {$\frac{3R}{4}$} radius}\

\begin{lemma}\label{lemma3.7}
Let $(M,g)$ be an $n$-dim $(n\ge3)$ complete manifold with $\mathrm{Ric}\ge-(n-1)\kappa$, where $\kappa$ is a non-negative constant. Furthermore, suppose that $a$, $q$, $p$ and $\beta$ satisfy the conditions stated in Lemma \ref{lemma3.6}. Let
$$\beta_1=\frac{n}{n-2}\cdot\frac{p+2t_0+2\alpha-2}{2}.$$
If $v$ is a positive solution to equation \eqref{1} on the geodesic ball $B(x_0,2R)\subset M$, then for $t_0$ large enough there exists $\mathcal{C}=\mathcal{C}(n,p,q,\beta)>0$ such that
\begin{align}
\begin{split}
\|F\|_{L^{\beta_1}\left(B\left(x_0,\frac{3}{4}R\right)\right)}\le&\mathcal{C}V^\frac{1}{\beta_1}\left(\kappa+\frac{1}{R^2} \right)\phi^{2-\beta},
\end{split}
\end{align}
where $V$ is the volume of geodesic ball $B_R(x_0)$, $\phi$ is defined in \eqref{3.78}.
\end{lemma}

\begin{proof}
We can infer from \eqref{3.79} that
\begin{align}\label{3.80}
\begin{split}
&\exp\left\{-C_n(1+\sqrt{\kappa}R)\right\}V^\frac{2}{n}R^{-2}\left\|F^{\frac{p}{4}+\frac{\alpha-1}{2}
+\frac{t_0}{2}}\eta\right\|^2_{L^\frac{2n}{n-2}(\Omega)}+\mu t_0\phi_\beta^{\beta-2}\int_\Omega F^{t_0+\alpha+\frac{p}{2}}\eta^2\\
\le&\left[(n-1)\mu_1t_0\kappa+ \frac{1}{R^2}\right] \int_\Omega F^{t_0+\alpha+\frac{p}{2}-1}\eta^2+\mu_1\int_\Omega F^{t_0+\frac{p}{2}+\alpha-1}|\nabla\eta|^2,
\end{split}
\end{align}
where $t=t_0$ satisfies \eqref{3.68}. Denote
\begin{align*}
{\Omega}_1=\left\{x:~F\ge\left[2(n-1)\kappa\frac{\mu_1}{\mu}+\frac{2}{\mu t_0 R^2}\right]\phi^{2-\beta},~x\in\Omega\right\},
\end{align*}
then we have
\begin{align}\label{3.81}
\left[(n-1)\mu_1t_0\kappa+ \frac{1}{R^2}\right] \int_{\Omega_1} F^{t_0+\alpha+\frac{p}{2}-1}\eta^2\le\frac{\mu}{2}t_0 \phi_\beta^{\beta-2}\int_\Omega F^{t_0+\alpha+\frac{p}{2}}\eta^2.
\end{align}
Set
\begin{align*}
{\Omega}_2=\Omega\setminus {\Omega}_1=\left\{x:~F<\left[2(n-1)\kappa\frac{\mu_1}{\mu}+\frac{2}{\mu t_0 R^2}\right]\phi^{2-\beta},~x\in\Omega\right\}.
\end{align*}
Then we have
\begin{align}\label{3.82}
\begin{split}
&\left[(n-1)\mu_1t_0\kappa+ \frac{1}{R^2}\right] \int_{\Omega_2} F^{t_0+\alpha+\frac{p}{2}-1}\eta^2\\
\le&\left[(n-1)\mu_1t_0\kappa+ \frac{1}{R^2}\right]\left\{\left[2(n-1)\kappa\frac{\mu_1}{\mu}+\frac{2}{\mu t_0 R^2}\right]\phi^{2-\beta}\right\}^{t_0+\alpha+\frac{p}{2}-1}V,
\end{split}
\end{align}
where $V$ is the volume of $\Omega=B(x_0,R)$. Combining \eqref{3.81} and \eqref{3.82} together, we obtain
\begin{align}\label{3.83}
\begin{split}
&\left[(n-1)\mu_1t_0\kappa+ \frac{1}{R^2}\right] \int_\Omega F^{t_0+\alpha+\frac{p}{2}-1}\eta^2-\frac{\mu}{2}t_0 \phi_\beta^{\beta-2}\int_\Omega F^{t_0+\alpha+\frac{p}{2}}\eta^2\\
\le&\left[(n-1)\mu_1t_0\kappa+ \frac{1}{R^2}\right]\left\{\left[2(n-1)\kappa\frac{\mu_1}{\mu}+\frac{2}{\mu t_0 R^2}\right]\phi^{2-\beta}\right\}^{t_0+\alpha+\frac{p}{2}-1}V.
\end{split}
\end{align}

We denote $\Omega_1=B(x_0,\frac{3R}{4})$ and choose $\eta_1\in C_0^\infty(\Omega)$ satisfying
\begin{align*}
\begin{cases}
0\le\eta_1\le1,\quad\eta_1 \equiv1\quad in~\Omega_1;\\
|\nabla\eta_1|\le\frac{C}{R},
\end{cases}
\end{align*}
and let
$$\eta=\eta_1^{t_0+\frac{p}{2}+\alpha}.$$
Then, we have
\begin{align}\label{3.84}
\begin{split}
\mu_1\int_\Omega F^{t_0+\frac{p}{2}+\alpha-1}|\nabla\eta|^2=&\mu_1\left(t_0+\frac{p}{2}+\alpha\right)^2\int_\Omega F^{t_0+\frac{p}{2}+\alpha-1}\eta_1^{p+2\alpha-2+2t_0}|\nabla\eta_1|^2\\
\le&\mu_1\left(t_0+\frac{p}{2}+\alpha\right)^2\frac{C^2}{R^2}\int_\Omega F^{t_0+\frac{p}{2}+\alpha-1}\eta_1^{p+2\alpha-2+2t_0}.
\end{split}
\end{align}
By H\"older's inequality, \eqref{3.84} can be written as
\begin{align}\label{3.85}
\begin{split}
\mu_1\int_\Omega F^{t_0+\frac{p}{2}+\alpha-1}|\nabla\eta|^2\le&\mu_1\left(t_0+\frac{p}{2}+\alpha\right)^2\frac{C^2}{R^2}\left(\int_\Omega F^{t_0+\frac{p}{2}+\alpha}\eta_1^{2t_0+p+2\alpha}\right)^\frac{2t_0+p+2\alpha-2}{2t_0+p+2\alpha}V^\frac{2}{2t_0+p+2\alpha}\\
=&\mu_1\left(t_0+\frac{p}{2}+\alpha\right)^2\frac{C^2}{R^2}\left(\int_\Omega F^{t_0+\frac{p}{2}+\alpha}\eta^2\right)^\frac{2t_0+p+2\alpha-2}{2t_0+p+2\alpha}V^\frac{2}{2t_0+p+2\alpha}.
\end{split}
\end{align}
By using Young's inequality, we can write \eqref{3.85} as
\begin{align}\label{3.86}
\begin{split}
&\mu_1\int_\Omega F^{t_0+\frac{p}{2}+\alpha-1}|\nabla\eta|^2\\
\le&\frac{\mu t_0}{2}\phi^{\beta-2}\int_\Omega F^{t_0+\frac{p}{2}+\alpha}\eta^2\\
&+\frac{2}{2t_0+p+2\alpha}\left[\mu_1\left(t_0+\frac{p}{2}+\alpha\right)^2\frac{C^2}{R^2}\right]^{t_0+\frac{p}{2}
+\alpha}\left[\frac{2(2t_0+p+2\alpha-2)}{(2t_0+p+2\alpha)\mu t_0}\phi^{2-\beta}\right]^{t_0+\frac{p}{2}+\alpha-1}V.
\end{split}
\end{align}
Substituting \eqref{3.83} and \eqref{3.86} into \eqref{3.80}  gives
\begin{align}\label{3.87}
\begin{split}
&\exp\left\{-C_n(1+\sqrt{\kappa}R)\right\}V^\frac{2}{n}R^{-2}\left\|F^{\frac{p}{4}+\frac{\alpha-1}{2}
+\frac{t_0}{2}}\eta\right\|^2_{L^\frac{2n}{n-2}(\Omega)}\\
\le&\left[(n-1)\mu_1t_0\kappa+ \frac{1}{R^2}\right]\left\{\left[2(n-1)\kappa\frac{\mu_1}{\mu}+\frac{2}{\mu t_0 R^2}\right]\phi^{2-\beta}\right\}^{t_0+\alpha+\frac{p}{2}-1}V\\
&+\frac{2}{2t_0+p+2\alpha}\left[\mu_1\left(t_0+\frac{p}{2}+\alpha\right)^2\frac{C^2}{R^2}\right]^{t_0+\frac{p}{2}+\alpha}
\left[\frac{2(2t_0+p+2\alpha-2)}{(2t_0+p+2\alpha)\mu t_0}\phi^{2-\beta}\right]^{t_0+\frac{p}{2}+\alpha-1}V.
\end{split}
\end{align}
Taking power of $\frac{2}{2t_0+p+2\alpha-2}$ on the both sides of $\eqref{3.87}$ respectively, we obtain
\begin{align*}
\begin{split}
&\|F\|_{L^{\beta_1}\left(B\left(x_0,\frac{3}{4}R\right)\right)}\\
\le&\left\|F^{\frac{p}{4}+\frac{\alpha-1}{2}
+\frac{t_0}{2}}\eta\right\|^\frac{4}{2t_0+p+2\alpha-2}_{L^\frac{2n}{n-2}(\Omega)}\\
\le&\exp\left\{\frac{2C_n(1+\sqrt{\kappa}R)}{2t_0+p+2\alpha-2}\right\}V^\frac{1}{\beta_1}\Bigg\{\left[(n-1)\mu_1t_0R^2\kappa+ 1\right]\left\{\left[2(n-1)\kappa\frac{\mu_1}{\mu}+\frac{2}{\mu t_0 R^2}\right]\phi^{2-\beta}\right\}^{t_0+\alpha+\frac{p}{2}-1}\\
&+\frac{2}{2t_0+p+2\alpha}\left[\mu_1\left(t_0+\frac{p}{2}+\alpha\right)^2\frac{C^2}{R^2}\right]^{t_0+\frac{p}{2}+\alpha}
\left[\frac{2(2t_0+p+2\alpha-2)}{(2t_0+p+2\alpha)\mu t_0}\phi^{2-\beta}\right]^{t_0+\frac{p}{2}+\alpha-1}R^2\Bigg\}^\frac{2}{2t_0+p+2\alpha-2},
\end{split}
\end{align*}
where
$$\beta_1=\frac{n}{n-2}\frac{p+2t_0+2\alpha-2}{2}.$$

By using the inequality
$$(a_1+a_2)^{b_1}\le 2^{b_1}(a_1^{b_1}+a_2^{b_1}),~\quad (a_i\ge0,~b_1>0),$$
we infer from the above inequality
\begin{align}\label{3.88}
\begin{split}
&\|F\|_{L^{\beta_1}\left(B\left(x_0,\frac{3}{4}R\right)\right)}\\
\le&\exp\left\{\frac{2C_n(1+\sqrt{\kappa}R)}{2t_0+p+2\alpha-2}\right\}V^\frac{1}{\beta_1}2^\frac{2}{2t_0+p+2\alpha-2}\\
&\cdot\Bigg\{\left[(n-1)\mu_1t_0R^2\kappa+ 1\right]^\frac{2}{2t_0+p+2\alpha-2}\left[2(n-1)\kappa\frac{\mu_1}{\mu}+\frac{2}{\mu t_0 R^2}\right]\phi^{2-\beta}\\
&+\left(\frac{2}{2t_0+p+2\alpha}\right)^\frac{2}{2t_0+p+2\alpha-2}\left[\mu_1\left(t_0+\frac{p}{2}
+\alpha\right)^2C^2\right]^{\frac{2t_0+p+2\alpha}{2t_0+p+2\alpha-2}}\frac{2(2t_0+p+2\alpha-2)}{(2t_0+p+2\alpha)\mu t_0}\phi^{2-\beta}R^{-2}\Bigg\}\\
:=&\exp\left\{\frac{2C_n(1+\sqrt{\kappa}R)}{2t_0+p+2\alpha-2}\right\}V^\frac{1}{\beta_1}(I_1+I_2).
\end{split}
\end{align}
For $I_1$, we have
\begin{align}\label{I_1}
\begin{split}
I_1=&2^\frac{2}{2t_0+p+2\alpha-2}\left[(n-1)\mu_1t_0R^2\kappa+ 1\right]^\frac{2}{2t_0+p+2\alpha-2}\left[2(n-1)\kappa\frac{\mu_1}{\mu}+\frac{2}{\mu t_0 R^2}\right]\phi^{2-\beta}\\
=&2^\frac{2t_0+p+2\alpha}{2t_0+p+2\alpha-2}\left[(n-1)\mu_1t_0R^2\kappa+1\right]^\frac{2t_0+p+2\alpha}{2t_0+p+2\alpha-2}t_0^{-1}\mu^{-1}R^{-2}\phi^{2-\beta}\\
\le&2^\frac{2t_0+p+2\alpha}{2t_0+p+2\alpha-2}\left[(n-1)\mu_1t_0+1\right]^\frac{2t_0+p+2\alpha}{2t_0+p+2\alpha-2}\left(1+R^2\kappa\right)^\frac{2t_0+p+2\alpha}{2t_0+p+2\alpha-2}t_0^{-1}\mu^{-1}R^{-2}\phi^{2-\beta}\\
=&2^\frac{2t_0+p+2\alpha}{2t_0+p+2\alpha-2}\left[(n-1)\mu_1+t_0^{-1}\right]^\frac{2}{2t_0+p+2\alpha-2}\left(1+R^2\kappa\right)^\frac{2t_0+p+2\alpha}{2t_0+p+2\alpha-2}\mu^{-1}R^{-2}\phi^{2-\beta}.\\
\end{split}
\end{align}
Noticing that
\begin{align*}
\lim_{t_0\to+\infty}{2^\frac{2t_0+p+2\alpha}{2t_0+p+2\alpha-2}\left[(n-1)\mu_1+t_0^{-1}\right]^\frac{2}{2t_0+p+2\alpha-2}}=2,
\end{align*}
we can verify that
\begin{align}\label{limit}
\begin{split}
\mathop{\sup}\limits_{t_0\in[1,+\infty)}2^\frac{2t_0+p+2\alpha}{2t_0+p+2\alpha-2}\left[(n-1)\mu_1+t_0^{-1}\right]^\frac{2}{2t_0+p+2\alpha-2}<+\infty.
\end{split}
\end{align}
Combining \eqref{I_1} and \eqref{limit} together leads to
\begin{align}\label{I_1o}
\begin{split}
I_1\le\mathcal{C}_{I_1}\left(1+R^2\kappa\right)^\frac{2t_0+p+2\alpha}{2t_0+p+2\alpha-2}\mu^{-1}R^{-2}\phi^{2-\beta},
\end{split}
\end{align}
where $\mathcal{C}_{I_1}$ is a positive constant which depends only on $n$, $p$ and $q$.

Similarly, we have
\begin{align}
\label{I_2}I_2&\le\mathcal{C}_{I_2}\mu^{-1}t_0R^{-2}\phi^{2-\beta},
\end{align}
where $\mathcal{C}_{I_2}$ is a positive constant which depends only on $n$, $p$ and $q$.

Substituting \eqref{I_1o} and \eqref{I_2} into \eqref{3.88}, we obtain
\begin{align*}
\|F\|_{L^{\beta_1}\left(B\left(x_0,\frac{3}{4}R\right)\right)}\le&\exp\left\{\frac{2C_n(1+\sqrt{\kappa}R)}{2t_0+p+2\alpha-2}\right\}
V^\frac{1}{\beta_1}\left[\mathcal{C}_{I_1}(1+\kappa R^2)^\frac{2t_0+p+2\alpha}{2t_0+p+2\alpha-2} +\mathcal{C}_{I_2}t_0\right]R^{-2}\mu^{-1}\phi^{2-\beta}.
\end{align*}
Set
\begin{align*}
\mathcal{C}_1=\max\left\{\mathcal{C}_{I_1},\mathcal{C}_{I_2}\right\}\mu^{-1}.
\end{align*}
Then we can rewritten the above inequality as
\begin{align}\label{3.93}
\|F\|_{L^{\beta_1}\left(B\left(x_0,\frac{3}{4}R\right)\right)}\le&\mathcal{C}_1\exp\left\{\frac{2C_n(1+\sqrt{\kappa}R)}{2t_0+p+2\alpha-2}\right\}
V^\frac{1}{\beta_1}\left[(1+\kappa R^2)^\frac{2t_0+p+2\alpha}{2t_0+p+2\alpha-2}+t_0\right]R^{-2}\phi^{2-\beta},
\end{align}
where $\mathcal{C}_1$ is a positive constant which depends only on $n$, $\beta$, $p$ and $q$.

Now, let $t_0$ satisfies \eqref{3.68} and
\begin{align}\label{3.94}
(1+\kappa R^2)\le t_0\le \mathcal{C}_0 (1+\kappa R^2),
\end{align}
where $\mathcal{C}_0=\mathcal{C}_0(n,p,q,\beta)$ is a positive constant. Hence, combining \eqref{3.93} and \eqref{3.94} together we have
\begin{align*}
\|F\|_{L^{\beta_1}\left(B\left(x_0,\frac{3}{4}R\right)\right)}\le&\mathcal{C}V^\frac{1}{\beta_1}\left(\kappa+\frac{1}{R^2} \right)\phi^{2-\beta},
\end{align*}
where $\mathcal{C}$ is a positive constant which depends only on $n$, $\beta$, $p$ and $q$. Hence, we complete the proof of Lemma \ref{lemma3.7}.
\end{proof}

\subsection{Moser iteration for positive solutions of \eqref{1}}\
\medskip

By using the integral inequality \eqref{3.79}, we can achieve the following lemma:
\begin{lemma}\label{lemma3.8}
Let $(M,g)$ be an $n$-dim $(n\ge3)$ complete manifold with $\mathrm{Ric}\ge-(n-1)\kappa$, where $\kappa$ is a non-negative constant. Furthermore, suppose that $a$, $q$, $p$ and $\beta$ satisfy the same conditions as in Lemma \ref{lemma3.6}. Let
$$\beta_1=\frac{n}{n-2}\frac{p+2t_0+2\alpha-2}{2}.$$
If $v$ is a positive solution to equation \eqref{1} on the geodesic ball $B(x_0,2R)\subset M$, then for $t_0$ large enough there exists $\mathcal{C}=\mathcal{C}(n,p,q)>0$ such that
\begin{align}
\|F\|_{L^{\infty}\left(B\left(x_0,\frac{R}{2}\right)\right)}\le\mathcal{C}V^{-\frac{1}{\beta_1}}\|F\|_{L^{\beta_{1}}\left(\Omega_{1}\right)},
\end{align}
where $V$ is the volume of geodesic ball $B_R(x_0)$.
\end{lemma}

\begin{proof}
By dropping the second non-negative term in \eqref{3.79}, we obtain
\begin{align}\label{3.96}
\begin{split}
&\exp\left\{-C_n(1+\sqrt{\kappa}R)\right\}V^\frac{2}{n}R^{-2}\left\|F^{\frac{p}{4}+\frac{\alpha-1}{2}
+\frac{t}{2}}\eta\right\|^2_{L^\frac{2n}{n-2}(\Omega)}\\
\le&\left[(n-1)\mu_1t\kappa+ \frac{1}{R^2}\right] \int_\Omega F^{t+\alpha+\frac{p}{2}-1}\eta^2 +\mu_1\int_\Omega F^{t+\frac{p}{2}+\alpha-1}|\nabla\eta|^2.
\end{split}
\end{align}
Now, we denote $$r_m=\frac{R}{2}+\frac{R}{4^m}$$ and $\Omega_m=B(x_0,r_m)$; and then choose $\eta_m\in C_0^\infty(\Omega_m)$ satisfying
\begin{align*}
\begin{cases}
0\le\eta_m\le1,\quad\eta_m\equiv1\quad in~B(x_0,r_{m+1});\\
|\nabla\eta_m|\le C\frac{4^m}{R}.
\end{cases}
\end{align*}
Substituting $\eta$ by $\eta_m$ in \eqref{3.96}, we can easily verify that
\begin{align}
\begin{split}\label{3.97}
&\exp\left\{-C_n(1+\sqrt{\kappa}R)\right\}V^\frac{2}{n}R^{-2}\left\|F^{\frac{p}{4}+\frac{\alpha-1}{2}
+\frac{t}{2}}\right\|^2_{L^\frac{2n}{n-2}(\Omega_{m+1})}\\
\le&\left[(n-1)\mu_1t\kappa+\frac{1}{R^2}\right]\int_{\Omega_m} F^{t+\frac{p}{2}+\alpha-1}+\mu_1\frac{C^216^{m}}{R^2}\int_{\Omega_m} F^{t+\frac{p}{2}+\alpha-1}.
\end{split}
\end{align}
Next, we choose
$$\beta_1=\left(t_0+\frac{p}{2}+\alpha-1\right)\frac{n}{n-2}$$
and $$\beta_{m+1}=\frac{n\beta_m}{n-2},$$
and let $t=t_m$ such that
$$t_m+\frac{p}{2}+\alpha-1=\beta_m.$$
Then it follows that
\begin{align}\label{3.98}
\begin{split}
\exp\left\{-C_n(1+\sqrt{\kappa}R)\right\}V^\frac{2}{n}\left(\int_{\Omega_{m+1}}F^{\beta_{m+1}}\right)^\frac{n-2}{n}\le
\left[(n-1)\mu_1t_mR^2\kappa+1+\mu_1{C^216^{m}}\right]\int_{\Omega_m} F^{\beta_m}.
\end{split}
\end{align}
Taking power of $\frac{1}{\beta_m}$ on the both sides of $\eqref{3.98}$, we obtain
\begin{align}
\begin{split}\label{3.99}
&\|F\|_{L^{\beta_{m+1}}\left(\Omega_{m+1}\right)}\\
\le&\exp\left\{\frac{C_n(1+\sqrt{\kappa}R)}{\beta_m}\right\}V^{-\frac{2}{n\beta_m}}\left[(n-1)\mu_1t_m\kappa R^2+1+\mu_1{C^216^{m}}\right]^\frac{1}{\beta_m}\|F\|_{L^{\beta_{m}}\left(\Omega_{m}\right)}.
\end{split}
\end{align}
Keeping the definition of $t_m$ in mind, from \eqref{3.99} we can deduce that
\begin{align*}
\begin{split}
\|F\|_{L^{\beta_{m+1}}\left(\Omega_{m+1}\right)}\le&\exp\left\{\frac{C_n(1+\sqrt{\kappa}R)}{\beta_m}\right\}
V^{-\frac{2}{n\beta_m}}16^{\frac{m}{\beta_m}}\\
&\cdot\left[(n-1)\mu_1\left(t_0+\frac{p}{2}+\alpha-1\right)\kappa R^2+1+\mu_1C^2\right]^\frac{1}{\beta_m}\|F\|_{L^{\beta_{m}}\left(\Omega_{m}\right)}.
\end{split}
\end{align*}
Noting
\begin{align*}
\sum_{m=1}^\infty{\frac{1}{\beta_m}}=\frac{n}{2\beta_1}\quad \quad \mbox{and}\quad\quad \sum_{m=1}^\infty{\frac{m}{\beta_m}}=\frac{n^2}{4\beta_1},
\end{align*}
we have
\begin{align*}
\begin{split}
\|F\|_{L^{\infty}\left(B\left(x_0,\frac{R}{2}\right)\right)}\le&\exp\left\{\frac{nC_n(1+\sqrt{\kappa}R)}{2\beta_1}\right\}
V^{-\frac{1}{\beta_1}}16^{\frac{n^2}{4\beta_1}}\\
&\cdot\left[(n-1)\mu_1\left(t_0+\frac{p}{2}+\alpha-1\right)\kappa R^2+1+\mu_1C^2\right]^\frac{n}{2\beta_1}\|F\|_{L^{\beta_{1}}\left(\Omega_{1}\right)}\\
\le&\mathcal{C}_3\exp\left\{\frac{nC_n(1+\sqrt{\kappa}R)}{2t_0+p+2\alpha-2}\right\}V^{-\frac{1}{\beta_1}}(1+\kappa R^2)^\frac{n}{2\beta_1}\|F\|_{L^{\beta_{1}}\left(\Omega_{1}\right)},\\
\end{split}
\end{align*}
where $\mathcal{C}_3=\mathcal{C}_3(n,p,q)$ is a positive constant. Since $t_0$ satisfies \eqref{3.79} and \eqref{3.94}, it is not difficult to see that
\begin{align*}
\|F\|_{L^{\infty}\left(B\left(x_0,\frac{R}{2}\right)\right)}\le\mathcal{C}V^{-\frac{1}{\beta_1}}\|F\|_{L^{\beta_{1}}\left(\Omega_{1}\right)},
\end{align*}
where $\mathcal{C}=\mathcal{C}(n,p,q)$ is a positive constant.
\end{proof}

\section{\textbf{Proof of main theorem}}
In fact, we can easily achieve Theorem \ref{theorem1.1} by using Lemma \ref{lemma3.7} and Lemma \ref{lemma3.8}. Therefore, we omit its proof.

\ \

\noindent\textbf{Proof of Corollary \ref{corollary1.2}}
\begin{proof}
By using Theorem \ref{theorem1.1}, we only need to confirm that the constants $a$, $q$ and $p$ satisfy one of the following two estimates
\begin{align*}
\frac{\beta}{2}\cdot\frac{n+1}{n-1}(p-1)-h^{\frac{1}{2}}<q<\frac{\beta}{2}\cdot\frac{n+1}{n-1}(p-1)+h^{\frac{1}{2}}\quad (a\neq 0);
\end{align*}
\begin{align}\label{4.1}
a\left[\frac{\beta}{2}\cdot\frac{n+1}{n-1}(p-1)-q\right]\ge0.
\end{align}
Here, we only check the case 1, the others are similar.

\textbf{Case 1}: $a>0$, $p\ge n$ and $q\in \mathbb{R}$.

Since $p\ge n$, we can see that $\beta\in (0,+\infty)$ by using \eqref{beta}. Furthermore, since $a>0$, we can verify that $\eqref{4.1}$ is equivalent to
\begin{align}\label{4.2}
q\le\frac{\beta}{2}\cdot\frac{n+1}{n-1}(p-1).
\end{align}
Hence, for any fixed $p$ ($p\ge n$), $n$ and $q\in\mathbb{R}$, we can make $\eqref{4.2}$ be true by letting $\beta$ large enough. Therefore, we complete the proof of case 1.
\end{proof}

The proof of Corollary \ref{corollary1.3} and Corollary \ref{corollary1.4} are similar to Corollary \ref{corollary1.2}. Hence, we omit their proofs.
\medskip

\noindent\textbf{Proof of Theorem \ref{corollary1.5}}
\begin{proof}
By using Corollary \ref{corollary1.2}, we can know that there exist positive constants ${\mathcal{C}}={\mathcal{C}}(n,p,q)$ and $\beta=\beta(n,p,q)\in(0,\infty)$, such that the following estimate holds true
\begin{align}\label{4.3}
\mathop{\sup}\limits_{B\left(x_0,\frac{R}{2}\right)}\frac{|\nabla v|^2}{v^\beta}\le\mathcal{C}\frac{1+\kappa R^2}{R^2}\phi_\beta^{2-\beta}.
\end{align}
Since
\begin{align}\label{4.4}
\begin{split}
\phi_\beta=
\begin{cases}
            \mathop{\sup}\limits_{B(x_0,R)}v,&0<\beta<2,
            \\[3mm]
            1,&\beta=2,
            \\[3mm]
            \mathop{\inf}\limits_{B(x_0,R)}v,&\beta>2,
\end{cases}
\end{split}
\end{align}
we consider the following three cases:

\textbf{Case 1}. $\beta\in(0,2)$.

Combining \eqref{4.3} and \eqref{4.4} together, we have the following estimate
\begin{align}\label{4.5}
\mathop{\sup}\limits_{B\left(x_0,\frac{R}{2}\right)}\frac{|\nabla v|^2}{v^\beta}\le\mathcal{C}\frac{1+\kappa R^2}{R^2}\bigg(\mathop{\sup}\limits_{B(x_0,R)}v\bigg)^{2-\beta}.
\end{align}
Multiplying the both sides of \eqref{4.5} by $\mathop{\sup}\limits_{B \left(x_0,R\right)}v^{\beta-2}$ leads to
\begin{align}\label{4.6}
\mathop{\sup}\limits_{B\left(x_0,\frac{R}{2}\right)}\frac{|\nabla v|^2}{v^\beta}\mathop{\sup}\limits_{B \left(x_0,R\right)}v^{\beta-2}\le\mathcal{C}\frac{1+\kappa R^2}{R^2}\bigg(\mathop{\sup}\limits_{B(x_0,R)}v\bigg)^{2-\beta}\mathop{\sup}\limits_{B \left(x_0,R\right)}v^{\beta-2}.
\end{align}
 Since
\begin{align*}
    \mathop{\sup}\limits_{B\left(x_0,R\right)}v\le l\mathop{\inf}\limits_{B\left(x_0,R\right)}v,
\end{align*}
we can see that
\begin{align}\label{4.7}
\begin{split}
\bigg(\mathop{\sup}\limits_{B(x_0,R)}v\bigg)^{2-\beta}\mathop{\sup}\limits_{B\left(x_0,R\right)}v^{\beta-2}=\bigg(\mathop{\sup}\limits_{B(x_0,R)}v\bigg)^{2-\beta} \bigg(\mathop{\inf}\limits_{B\left(x_0,R\right)}v\bigg)^{\beta-2}\le{l^{2-\beta}}.
\end{split}
\end{align}
Furthermore,
\begin{align}\label{4.8}
\mathop{\sup}\limits_{B\left(x_0,\frac{R}{2}\right)}\frac{|\nabla v|^2}{v^2}\le\mathop{\sup}\limits_{B\left(x_0,\frac{R}{2}\right)}\frac{|\nabla v|^2}{v^\beta}\mathop{\sup}\limits_{B \left(x_0,\frac{R}{2}\right)}v^{\beta-2}\le\mathop{\sup}\limits_{B\left(x_0,\frac{R}{2}\right)}\frac{|\nabla v|^2}{v^\beta}\mathop{\sup}\limits_{B \left(x_0,R\right)}v^{\beta-2}.
\end{align}
Now, substituting \eqref{4.7} and \eqref{4.8} into \eqref{4.6} leads to
\begin{align}
\mathop{\sup}\limits_{B\left(x_0,\frac{R}{2}\right)}\frac{|\nabla v|^2}{v^2}\le\mathcal{C}l^{2-\beta}\frac{1+\kappa R^2}{R^2}.
\end{align}
Thus we finish the proof of case 1.
\medskip

\textbf{Case 2}. $\beta=2.$

Combining \eqref{4.3} and \eqref{4.4} together, we have the following estimate
\begin{align}
\mathop{\sup}\limits_{B\left(x_0,\frac{R}{2}\right)}\frac{|\nabla v|^2}{v^2}\le\mathcal{C}\frac{1+\kappa R^2}{R^2}.
\end{align}
Therefore, we complete the proof of case 2.
\medskip

\textbf{Case 3}. $\beta>2.$

The proof of case 3 is similar to case 1. Hence, we omit its proof.

\end{proof}
\medskip

\noindent\textbf{Proof of Corollary \ref{corollary1.5*}:} The proof of this Corollary is obvious and we omit it.
\medskip

Now, we turn to considering \eqref{b=0} in $\mathbb{R}^n$ and give the proof of Corollary \ref{corollary1.6}. Here, we need to make use of Lemma \ref{lemma2.4} and Lemma \ref{lemma2.5} in the following proof.
\medskip

\noindent\textbf{Proof of Corollary \ref{corollary1.6}}
\begin{proof}
In order to use Lemma \ref{lemma2.4} and Lemma \ref{lemma2.5}, we need to consider the following four cases:


\textbf{Case 1.} $a>0$, $1<p<n$, $p\neq q$ and $q\in\left(p-1,~\frac{(p-1)n}{n-p}\right)$.

Let $v=a^{\frac{1}{p-q}}w$, then we have the following equality about $w$
\begin{align}
\Delta_pw+w^q=0,
\end{align}
by using \eqref{b=0}. By using Lemma \ref{lemma2.4}, we can know that for any $1<p<n$ and
$$q\in\left(p-1,~\frac{(p-1)n}{n-p}\right),$$
the following Harnack inequality
\begin{align}\label{4.12}
\mathop{\sup}\limits_{B_R}w(x)\le\mathcal{C}\mathop{\inf}\limits_{B_R}w(x)
\end{align}
holds true, where $\mathcal{C}=\mathcal{C}(n,p,q)$ is a positive constant. Combining Theorem \ref{corollary1.5} and \eqref{4.12} together, we can deduce that for any $1<p<n$ and
$$q\in\left(p-1,~\frac{(p-1)n}{n-p}\right),$$
the following estimate
\begin{align}\label{4.13}
\mathop{\sup}\limits_{B\left(x_0,\frac{R}{2}\right)}\frac{|\nabla w|^2}{w^2}\le\frac{\mathcal{C}}{R^2}
\end{align}
holds true, where $\mathcal{C}=\mathcal{C}(n,p,q)$ is a positive constant. Substituting $w=a^{\frac{1}{q-p}}v$ into \eqref{4.13}, we have the following estimate
\begin{align*}
\mathop{\sup}\limits_{B\left(x_0,\frac{R}{2}\right)}\frac{|\nabla v|^2}{v^2}\le\frac{\mathcal{C}}{R^2},
\end{align*}
where $\mathcal{C}=\mathcal{C}(n,p,q)$ is a positive constant. Thus we finish the proof of case 1.
\medskip

\textbf{Case 2.} $a>0$, $p\ge n$, $p\neq q$ and $q\in\left(0,+\infty\right)$.

The proof of case 2 is similar to case 1. Hence, we omit its proof.
\medskip

\textbf{Case 3.} $a\ge 1$ and $1<p=q<n<p^2$.

Since $p=q$ and $1<p<n<p^2$, we can know that $q\in\left(p-1,~\frac{(p-1)n}{n-p}\right)$. By using Lemma \ref{lemma2.4}, we know that for any $1<p=q<n<p^2$ the following Harnack inequality
\begin{align}\label{4.14}
\mathop{\sup}\limits_{B_R}v(x)\le\mathcal{C}\mathop{\inf}\limits_{B_R}v(x)
\end{align}
holds true, where $\mathcal{C}=\mathcal{C}(n,p,q,a)$ is a positive constant. By using Theorem \ref{corollary1.5} and \eqref{4.14}, we can achieve the following estimate
\begin{align*}
    \mathop{\sup}\limits_{B\left(x_0,\frac{R}{2}\right)}\frac{|\nabla v|^2}{v^2}\le\frac{\mathcal{C}}{R^2},
\end{align*}
where $\mathcal{C}=\mathcal{C}(n,p,q,a)$ is a positive constant. Thus we finish the proof of case 3.
\medskip

\textbf{Case 4.} $a\ge 1$ and $p=q\ge n$.

The proof of case 4 is similar to case 3. Therefore, we omit its proof.
\end{proof}

\noindent\textbf{Proof of Theorem \ref{theorem1.7}}
\begin{proof}
By using Corollary \ref{corollary1.4}, we can know that if $b>0$ and the constants $a$, $q$ and $p$ satisfy one of the following four conditions
\begin{itemize}
\item $a>0$, $p\ge n$ and $q\in \mathbb{R}$;
\item $a>0$, $1<p<n$ and $q<\Psi\left(\left[2,~\frac{2(n-1)}{n-p}\right)\right)$;
\item $a<0$, $p\ge n$ and $q>\Gamma([2,~+\infty))$;
\item $a<0$, $1<p<n$ and $q>\Gamma\left(\left[2,~\frac{2(n-1)}{n-p}\right)\right)$,
\end{itemize}
then the following estimate
\begin{align*}
\mathop{\sup}\limits_{B\left(x_0,\frac{R}{2}\right)}\frac{|\nabla v|^2}{(v+b)^\beta}\le\frac{\mathcal{C}}{R^2}
\end{align*}
holds true, where $\mathcal{C}=\mathcal{C}(n,p,q)$ is a positive constant. By letting $R\longrightarrow +\infty$, we can know that
\begin{align*}
\mathop{\sup}\limits_{M}|\nabla v|=0.
\end{align*}
Hence, $v$ is a constant. But, any positive constant is not a positive solution to \eqref{1} except for $a=0$.
\end{proof}

\noindent\textbf{The proof of Theorem \ref{theorem1.8}:} The proof of Theorem \ref{theorem1.8} is similar to Theorem \ref{theorem1.7}. Hence, we omit it.
\medskip
\medskip

\noindent {\it\bf{Acknowledgements}}: The authors are supported by National Natural Science Foundation of China (Grant No. 12431003).
\medskip
\medskip

\label{appendixB}

\end{document}